\documentclass[a4paper,12pt]{amsart}
\usepackage[utf8]{inputenc}

\usepackage{mathrsfs}

\def\ep{\varepsilon}
\def\al{\alpha}
\def\om{\omega}
\def\omb{\overline{\omega}}
\def\Om{\Omega}
\def\loc{_{{\rm loc}}}
\def\N{\mathbb{N}}
\def\H{\mathscr{H}^{-1}(\Om)}

\def\bm{\mathcal{BM}(\Om)}
\def\bmp{\mathcal{BM}_+(\Om)}
\def\bmpb{\mathcal{BM}_+(\overline\Om)}
\def\bmb{\mathcal{BM}(\overline\Om)}
\def\L{L^2(\Om)+\xspace}
\def\X{\mathbf{X}}
\def\B{\mathbf{B}}
\def\Y{\mathbf{Y}}
\def\w{\mathbf{w}}
\def\tu{\widetilde{u}}
\def\dt{\frac{\text{d}}{\text{d}t}}
\def\xspace{\mathscr{X}}
\def\yinf{\mathscr{Y}^\infty}
\def\yinfb{\overline{\mathscr{Y}}^\infty}
\def\bcs{C^{\sigma,\infty}_{c,{\rm tan}}(\overline{\Omega})}
\usepackage{amssymb}

\newtheorem{theorem}{Theorem}[section]
\newtheorem{lemma}[theorem]{Lemma}
\newtheorem{definition}[theorem]{Definition}
\newtheorem{proposition}[theorem]{Proposition}
\newtheorem{corollary}[theorem]{Corollary}
\newtheorem{remark}[theorem]{Remark}
\numberwithin{equation}{section} 
\newcommand{\real}{\mathbb{R}}

\newcommand{\be}{\begin{equation}}
\newcommand{\ee}{\end{equation}}
\renewcommand{\leq}{\leqslant}
\renewcommand{\geq}{\geqslant}
\newcommand{\til}{\widetilde}

\newcommand\cO{\mathscr{O}}

\newcommand{\ds}{\displaystyle}

\def\R{\mathbb{R}}
\def\d{\,{\rm d}}
 
\def\bX{{\mathbf X}}

\DeclareMathOperator\curl{curl}
\DeclareMathOperator\dive{div}
\DeclareMathOperator\supp{supp}

\begin{document} 
\title[Weak vorticity formulation in domains with boundary]{Weak vorticity formulation for the incompressible 2D Euler equations in domains with boundary} 
\author[Iftimie, Lopes Filho, Nussenzveig Lopes and Sueur]{{D. Iftimie,}
{M. C. Lopes Filho,}  
{H. J. Nussenzveig Lopes and}
{F. Sueur}}

\begin{abstract}
	In this article we examine the interaction of incompressible 2D flows with compact material boundaries. Our focus is the dynamic behavior of the
circulation of velocity around boundary components and the possible exchange between flow vorticity and boundary circulation in flows with vortex sheet
initial data. We formulate our results for flows outside a finite number of smooth obstacles. Our point of departure is the observation that ideal flows 
with vortex sheet regularity have well-defined circulations around connected components of the boundary. In addition, we show that the velocity can be uniquely 
reconstructed from the vorticity and boundary component circulations, which allows to recast 2D Euler evolution using vorticity and the circulations as dynamic 
variables. The weak form of this vortex dynamics formulation of the equations is called the {\em weak vorticity formulation}. Our first result is existence 
of a solution for the weak velocity formulation with vortex sheet initial data for flow outside a finite number of smooth obstacles. The proof is a straightforward adaptation of Delort's original existence result and requires the usual sign condition. The main result in this article is the equivalence between the weak velocity and weak vorticity formulations, without sign assumptions. Next, we focus on weak solutions obtained by mollifying initial data and passing to the limit, with the portion of vorticity singular with respect to the Lebesgue measure assumed to be nonnegative. For these solutions we prove that the circulations around each boundary component cannot be smaller than the initial data circulation, so that nonnegative vorticity may be absorbed by the boundary, but not produced by the boundary. In addition, we prove that if the weak solution conserves circulation at the boundary components it is a {\it boundary coupled weak solution}, a stronger version of the weak vorticity formulation. We prove existence of a weak solution which conserves circulation at the boundary components if the initial vorticity is integrable, i.e. if the singular part vanishes. In addition, we discuss the definition of the mechanical force which the flow exerts on material boundary components and its relation with conservation of circulation. Finally, we describe the corresponding results for a bounded domain with holes, and the adaptations required in the proofs.      
\end{abstract}

\maketitle
\tableofcontents

\section{Introduction}

For two-dimensional incompressible fluid flow, a vortex sheet is a curve where the velocity 
of the fluid has a tangential discontinuity. Vortex sheets are an idealized model of a thin region 
where the fluid is subjected to intense, strongly localized shear. Flows with vortex sheets are 
of critical physical interest in fluid mechanics for several reasons, specially because such flows 
are common in situations of practical interest, such as in the wake of an airfoil. Thus, the 
mathematical description of vortex sheet motion is a classical topic in fluid dynamics.

In 1990,  Delort proved global-in-time existence of weak solutions for the incompressible 
Euler equations having, as initial vorticity, a compactly supported, bounded Radon measure with distinguished 
sign in $H^{-1}$, plus an arbitrary, compactly supported, $L^p$ function, with $p>1$, see \cite{Delort91}. 
This includes a large class of examples in which the initial vorticity 
is actually supported on a curve (classical vortex sheets). This result was later extended to certain 
symmetric configurations of vorticity with sign change, see \cite{LNX01,LNX06}. Very little is known regarding 
Delort's weak solutions beyond their existence. Some interesting open questions are the conservation of 
kinetic energy, conservation of the total variation of vorticity, and the behavior of the support of vorticity. 

The original work by Delort included flows in the full plane, in bounded domains and in compact manifolds without 
boundary. The proof was based on a compensated-compactness argument for certain quadratic expressions in the components
of velocity in order to pass to the limit in the weak formulation of the momentum equations along an approximate 
solution sequence obtained by mollifying initial data and exactly solving the equations.

There is a large literature directly associated with Delort's Theorem. The convergence to a weak solution was extended
to approximations obtained by vanishing viscosity, see \cite{majda}, numerical approximations, see \cite{LiuXin,Schochet2}
and Euler-$\alpha$, see \cite{BLT10}. The initial data class was extended to the limiting case $p=1$, see \cite{SemEDPEcolePolytechniques,VecchiWu93},
and an alternative proof using harmonic analysis was produced, see \cite{EvansMuller}. In 1995, S. Schochet presented a simplified
proof of the full-plane case by introducing the {\it weak vorticity formulation} of the Euler equations, where the compensated compactness
argument at the heart of the original result becomes an elementary algebraic trick, see \cite{Schochet95}. 

This ``weak vorticity formulation'' of the vortex sheet initial data problem is the focus of the present work. Originally, this weak formulation
of the vortex sheet initial data problem was of interest for the simplification it provided of Delort's original argument, for enabling the 
extensions to some symmetric, sign-changing initial data in \cite{LNX01,LNX06} and for the general physical relevance of the vortex dynamics 
point of view in incompressible fluid dynamics. Recently, with the discovery and rapid development of the theory of wild solutions of the Euler equations
by De Lellis and Szekelyhidi, see \cite{BT13,DS12} and references therein, it became clear that the weak form of the momentum formulation of the Euler 
equations is severely incomplete. However, the extension of the theory of wild solutions to weak vorticity formulations of the Euler equations is an 
important open problem, which suggests that there may be additional information encoded on the vortex dynamics which would be very interesting to uncover.       

In the present article, we are interested in ideal flows in domains with boundary. We observe first that such flows, in themselves, are unphysical. 
Indeed, all ideal flows are, in fact, slightly viscous, and it has been known since Prandtl, see \cite{Prandtl}, that, in the presence of rigid boundaries, ideal flows and slightly viscous flows behave very differently. However, the mathematical study of ideal flows in domains with boundary has some  physical relevance, first because the inviscid flow serves as a natural scaffolding for constructing slightly viscous flow perturbatively, and, second, because the vanishing viscosity limit for incompressible flows in a domain with boundary is an important open problem, which could, in principle, lead to some very irregular solutions of 
the inviscid equations. We refer the reader to \cite{BT13} for a broad discussion of this problem and its relation to turbulence modelling. Perhaps the main qualitative feature of the way in which slightly viscous flow interacts with a rigid boundary is vorticity production and shedding. In general, ideal flows cannot exchange vorticity with a wall, but, as we will see, at the level of regularity of vortex sheets, this becomes an interesting possibility. The weak vorticity
formulation, as originally proposed by Schochet in \cite{Schochet95}, applied to full plane flows only. In this article we adapt this notion to flows in domains with boundary
in a way which provides a quantitative accounting of the possible vorticity production through the interaction of irregular inviscid flow with a rigid wall.  	            
In \cite{LNX01}, Lopes Filho, Nussenzveig Lopes and Xin introduced the weak vorticity formulation for half-plane flows. The weak vorticity
formulation makes use of the Biot-Savart law, which is explicit in the full plane and in the half-plane. In \cite{LNX01} the notion
of {\it boundary coupled weak solution} was also introduced, as a necessary and sufficient condition for the validity of the method of images. This work was later extended to compactly supported perturbations of the half-plane in \cite{LNX06}.

The point of departure for the present work is to look for weak solutions of the incompressible Euler equations in exterior domains with vortex sheet
initial data. Here, an exterior domain is the complement of a finite number of smooth, disjoint, obstacles. Existence of such weak solutions can be easily established - we show that Delort's original proof can be immediately adapted to the exterior domain. We then ask ourselves - what about a weak vorticity formulation?  The key
new issue in the exterior domain is the topology. A vortex dynamical description of ideal flow in domains with topology requires us to keep track of 
velocity circulation in the 1-homology classes of the domain, a family of conserved quantities for smooth flow, due to Kelvin's Circulation Theorem. 
We prove that one can make sense and keep track of these velocity circulations for flows with vortex sheet regularity.  

Even for smooth flows, reducing ideal fluid dynamics  to vortex dynamics in a domain with holes is a rather recent development. 
A full description, for the case of bounded domains, was developed by Flucher and Gustafsson in \cite{FG1997}, see also 
\cite{LopesFilho2007} for a more explicit account. Part of the work of the present paper is to adapt and extend the vortex dynamics formulation
of \cite{FG1997} to exterior domains and recast it as a weak solution theory which includes vortex sheet flows.

The remainder of this paper is divided as follows. In Sections 2 and 3 we put together the basic notation and description of smooth vortex
dynamics in an exterior domain. In Section 4 we adapt this machinery to the weak solution context. The main point is that circulation of
velocity around connected components of the boundary is well-defined, as long as velocity is at least locally integrable and vorticity is
a bounded Radon measure. Still in Section 4, we develop a suitable approximation result. In Section 5 we state and prove an adaptation of Delort's
existence result to exterior domain flow. In Section 6 we state and prove the main result of this article -- the equivalence between the weak velocity
and weak vorticity formulations of the 2D Euler equations. In Section 7 we derive some properties of weak solutions which arise as limits of exact
solutions with mollified initial data, we discuss the connection of boundary-coupled weak solutions and conservation of circulation on connected 
components of the boundary, and we prove existence of boundary-coupled weak solutions with integrable initial vorticity. In Section 8 we extend the
equivalence between weak velocity and weak vorticity formulations to an equivalence between their boundary-coupled versions. This implies that
the mechanical coupling between connected components of the boundary and the fluid flow is well-defined if and only if the solution is 
boundary-coupled. In Section 9 we describe the adaptation of our results to bounded domains with holes, a simpler situation than the exterior 
domain, we derive conclusions and we propose some open problems. Finally, in the Appendix we prove estimates for the Green's function and the 
Biot-Savart kernel on bounded domains with holes which, although not new, do not appear to be easily available in the literature.

\section{Harmonic vector fields and the Green's function}

In this section we introduce basic notation and collect a few estimates which will be needed throughout the paper.

Let $\Omega \subset \real^2$  be a smooth open set such that the complement of $\overline\Om$ is the union of 
$k  \in  \mathbb{N}^*$ disjoint, simply connected, smooth, open sets $\Omega_1,\dots,\Omega_k$. Let  $\Gamma_i$ be the boundary of $\Om_i$ so that $\partial\Om=\Gamma_1\cup\dots\cup\Gamma_k$.

We will assume, for convenience, throughout this paper, that one of the obstacles, say $\Omega_1$, contains the unit ball $B(0;1)$; 
this can be done without loss of generality,  incorporating a translation and dilation in all the proofs, if needed. We also introduce notation which will be used hereafter: $i=i(x)=(x_1,-x_2)/|x|^2$, for $x = (x_1,x_2)$. Let $i(\Om)\cup\{0\} = \til{\Om}$. It is easy to see that $i(\real^2 \setminus \overline{\Omega}_1) \cup \{0\}$ is a simply connected subset of $B(0;1)$ which contains $i(\Omega_j)$ for all $j=1,\ldots,k$. Therefore $\til{\Om}$ is a subset of $i(\real^2 \setminus \overline{\Omega}_1)\cup\{0\}$ with $k-1$ holes; the outer boundary of $\til{\Omega}$ is 
$i( \Gamma _1)$.

We denote by $G=G(x,y)$ the Green's function of the Laplacian in $\Om$ with Dirichlet boundary conditions. We also introduce the function $K=K(x,y) \equiv \nabla^\perp_x G(x,y)$, where $\nabla^{\perp}_x = (-\partial_{x_2},\partial_{x_1})$,  known as the kernel of the Biot-Savart law.

We will denote by $\yinf$ the space of functions $f\in C^\infty(\overline\Om)$ with bounded support and such that $f$ is constant in a neighborhood of each $\Gamma_j$ (with a constant depending on $j$). We also denote by $\yinfb$ the space of functions $f\in C^\infty(\overline\Om)$ with bounded support and such that $f$ is constant on each $\Gamma_j$ (with a constant depending on $j$). The space of smooth divergence free vector fields compactly supported in $\Om$ is denoted by $C^\infty_{c,\sigma}(\Om)$.

Next, we introduce the {\it harmonic measures} $\w_j$, $j=1,\ldots,k$, in $\Om$. These are solutions of the boundary-value problem:
\begin{equation*} 
\left\{
\begin{array}{ll}
\Delta \w _j = 0, &    \mbox{ in } \Omega, \\
\w _j = \delta_{j\ell},&     \mbox{ on } \Gamma _{\ell}, \ \ell = 1, \ldots, k.\\
\w _j \text{ has a finite limit at }\infty.
\end{array}
\right.
\end{equation*}
Observe that this system can be viewed as a Dirichlet problem for the Laplacian in $\Om\,\cup\,\{\infty\}$. Existence and uniqueness of the harmonic measures is well-known, and they can be expressed by means of an explicit formula in terms of the Green's function:
\begin{equation*}
\w _j(y)=-\frac1{2\pi}\int_{\Gamma_j}\frac{\partial G(x,y)}{\partial \hat n_x}\d S_x  
\end{equation*}
where $\hat n_x$ is the exterior unit normal vector at $\partial\Om$ (see, for example, Chapter 1, Section 10 of \cite{MR0045823}).

In what follows we adopt the convention $(a,b)^{\perp} = (-b,a)$. We denote by $\hat{\tau}$ the unit tangent vector to $\partial\Om$ oriented in the counterclockwise direction, \textit{i.e.} $\hat{\tau} = - \hat{n}^{\perp}$.

In \cite{Kikuchi83} K. Kikuchi constructed a special basis for the harmonic vector fields in $\Omega$ (i.e., vector fields which are both solenoidal and irrotational and which are tangent to the boundary), generators of the homology of $\Omega$, $\bX_j, \, 
j = 1, \ldots k.$ It was shown in \cite[Lemma 1.5]{Kikuchi83} that there exist harmonic functions $\Psi_j$, $j = 1,\ldots \ell$, such that 
\begin{equation} \label{propsPsij}
\begin{array}{l}
\mathrm{(i)} \;\;\bX_j = \nabla^{\perp}\Psi_j,\\ \\
\mathrm{(ii)}\;\; \Psi_j (x) = \frac{1}{2\pi}\log|x| + \cO(1), \text{ as } |x| \to \infty,\\ \\
\mathrm{(iii)}\;\; \int_{\Gamma _{\ell}} \bX_j \cdot \hat{\tau} \d S = \delta_{j\ell},\\ \\
\mathrm{(iv)} \;\;\Psi_j\big|_{\Gamma _{\ell}} = c_{j\ell}, \text{ where the $c_{j\ell}$ are constants},\\ \\
\mathrm{(v)} \;\;|\bX_j(x)| \leq \ds{\frac{C}{|x|}}.
\end{array}
\end{equation}
That these harmonic vector fields are uniquely defined was proved in \cite[Lemma 2.14]{Kikuchi83}.

In addition, since the harmonic vector fields can be represented as holomorphic functions  (if $X=(X_1,X_2)$ is a harmonic vector field, then $X_2 + i X_1$ is
holomorphic), we may  use a Laurent expansion
at infinity to deduce that 
\begin{equation} \label{harumph1} 
X_j (x) = \frac{x^{\perp}}{2 \pi |x|^2} + \mathcal{O}(|x|^{-2}), \mbox{ as } |x| \to \infty.
\end{equation}

We will denote by $\xspace=\langle \X_1,\dots,\X_k\rangle$ the vector space spanned by the harmonic vector fields $\X_1,\dots,\X_k$.

We end this section with two estimates involving the kernel of the Biot-Savart law and the Green's function.
\begin{proposition}\label{propK}
There exists a constant $M_1>0$ depending only on $\Om$ such that
\begin{equation}\label{estk}
|K(x,y)|\leq M_1\frac{|y|}{|x||x-y|}
\end{equation}
and
\begin{equation}\label{estg}
|G(x,y)|\leq M_1+M_1\left| \log\frac{|x||y|}{|x-y|}\right|.  
\end{equation}

\end{proposition}
\begin{proof}
Let $G_{\til{\Om}}$ denote the Green's function for $\til{\Om}$. With complex notation we have that $i(z)=1/z$, so $i$ is an holomorphic function. It is well known that the Green's functions of two conformally equivalent domains are obtained via composition with the conformal mapping between the two domains. Therefore the Green's function for $\Om$ is $G(x,y)= G_{\til{\Om}}(i(x),i(y))$. By construction, 
\begin{equation*} 
K(x,y) = \nabla^{\perp}_x G(x,y)= -Di(x)K_{\til{\Om}}(i(x),i(y)),
\end{equation*}
where $K_{\til{\Om}}$ denotes the kernel of the Biot-Savart law for $\til{\Om}$. We will use the following estimate:
\begin{equation} \label{thisisnotgood}
|K_{\til{\Om}}(\til{x},\til{y})| \leqslant \frac{C}{|\til{x}-\til{y}|} ,
\end{equation}
valid for $\til{x}$, $\til{y}\in\til{\Om}$ and some constant $C>0$ which depends only on $\til{\Om}$. The proof of this estimate is included in the Appendix, see Proposition \ref{appGreenAndBS}. 
We deduce, from \eqref{thisisnotgood}, the following pointwise estimate for the Biot-Savart kernel:
\begin{equation*} 
|K(x,y)| \leqslant \frac{C}{|x|^2|i(x)-i(y)|}=\frac{C|y|}{|x||x-y|}. 
\end{equation*}
Similarly, estimate \eqref{estg} follows from the bound
\begin{equation*}
 |G_{\til{\Om}}(\til{x},\til{y})| \leqslant C(1+|\log(|\til{x}-\til{y}|)|)  
\end{equation*}
which is also included in Proposition \ref{appGreenAndBS}.
\end{proof}
\begin{proposition}\label{propestK}
There exists a constant $M_2$ depending only on $\Om$ such that
\begin{equation}\label{estK}
|f(x)\cdot K(x,y)+f(y)\cdot K(y,x)|\leq M_2\|f\|_{W^{1,\infty}(\Om)}
\quad\forall x,y\in\Omega,\ x\neq y.  
\end{equation}
for every vector valued function $f\in W^{1,\infty}(\overline\Om;\R^2)$ whose restriction to $\partial\Om$ is normal to the boundary. In particular we may take $f = \nabla \varphi$, $\varphi \in \yinfb$.
\end{proposition}
\begin{remark}
  The case when $f$ vanishes on the boundary $\partial\Om$ was considered in \cite{LNX06}. 
\end{remark}
\begin{proof}
We observe first from \eqref{estk} that there exists a constant $M_3=M_3(\Om)$ such that
\begin{equation}
  \label{estK1}
|K(x,y)|\leq M_3+\frac{M_3}{|x-y|}  \quad\forall x,y\in\Omega,\ x\neq y.   
\end{equation}

Define
\begin{equation*}
  H(x,y)=G(x,y)-\frac1{2\pi}\log|x-y|.
\end{equation*}
We have that the function $H$ is harmonic in both its arguments on $\Omega\times\Omega$. Therefore
\begin{equation*}
K(x,y)+K(y,x)=\nabla_x^\perp G(x,y)+ \nabla_y^\perp G(y,x)=\nabla_x^\perp H(x,y)+ \nabla_y^\perp H(y,x) 
\end{equation*}
is also harmonic and smooth on $\overline\Omega\times\overline\Omega\setminus\{(x,x);\ x\in\partial\Omega\}$. Moreover, given that $G(x,y)$ vanishes when $x$ or $y$ belongs to $\partial\Om$, we also have that
\begin{equation}\label{Kbord}
  K(x,y)=0\quad \forall x\in\Om,\ y\in\partial\Om
\qquad\text{and}\qquad K(x,y)\text{ tangent to }\partial\Om \quad \forall x\in\partial\Om,\ y\in\overline\Om\setminus\{x\}.
\end{equation}

We bound
\begin{equation}\label{estK2}
\begin{split}
 |f(x)\cdot K(x,y)+f(y)\cdot &K(y,x)|
\leq|f(x)\cdot [K(x,y)+K(y,x)]|+|[f(x)-f(y)]\cdot K(y,x)|\\
&\leq|f(x)\cdot [K(x,y)+K(y,x)]|+2M_3\|f\|_{L^\infty(\Om)}+M_3\|\nabla f\|_{L^\infty(\Om)}
\end{split}
\end{equation}
where we used \eqref{estK1}.

Next, let us fix $x\in\Om$. The function
\begin{equation*}
  y\mapsto f(x)\cdot [K(x,y)+K(y,x)]
\end{equation*}
is harmonic in $\Om$ and smooth on $\overline\Om$. By the maximum principle we therefore have that
\begin{multline}\label{estK3}
\sup_{y\in\Om} |f(x)\cdot [K(x,y)+K(y,x)]|\\
\leq \max\bigl(\limsup_{|y|\to\infty} |f(x)\cdot [K(x,y)+K(y,x)]|,
\max_{y\in\partial\Om}|f(x)\cdot [K(x,y)+K(y,x)]|\bigr).
\end{multline}
Using \eqref{estK1} we have that
\begin{equation}\label{estK4}
\limsup_{|y|\to\infty} |f(x)\cdot [K(x,y)+K(y,x)]|\leq   2M_3\|f\|_{L^\infty(\Om)}.
\end{equation}
Moreover, using that $f$ is normal to the boundary $\partial\Om$ we have from \eqref{Kbord} that for $y\in\partial\Om$
\begin{multline}\label{estK5}
 |f(x)\cdot [K(x,y)+K(y,x)]| 
=|f(x)\cdot K(y,x)|
=|[f(x)-f(y)]\cdot K(y,x)|\\
\leq 2M_3\|f\|_{L^\infty(\Om)}+M_3\|\nabla f\|_{L^\infty(\Om)}.
\end{multline}

We conclude from relations \eqref{estK2}--\eqref{estK5}  that
\begin{equation*}
 |f(x)\cdot K(x,y)+f(y)\cdot K(y,x)|
\leq  4M_3\|f\|_{L^\infty(\Om)}+2M_3\|\nabla f\|_{L^\infty(\Om)}.
\end{equation*}
This completes the proof of the proposition.
\end{proof}

We end this section with the following convergence result.
\begin{lemma}\label{weakconvmeas}
Let $X$ be a metric space, locally compact and $\sigma$-compact. Let $(\mu_n)_{n\in\N}$ be a tight sequence of bounded measures converging weakly to a measure $\mu$. Suppose that  $(|\mu_n|)_{n\in\N}$ converges weakly to another measure $\nu$. 
Then for any bounded borelian function $f$ continuous outside a $\nu$-negligible set, we have that
$$
\lim_{n\to\infty}\int f\,d\mu_n=\int f\,d\mu.
$$
\end{lemma}
\begin{remark}
This lemma is almost the same as \cite[Lemma 6.3.1]{Chemin}. The difference here is that we have the additional assumption of tightness for the sequence of measures while allowing for test functions $f$ that does not necessarily vanish at infinity as required in \cite[Lemma 6.3.1]{Chemin}. Recall that a sequence of measures $\mu_n$ is called tight if, for any $\ep>0$, there is a compact subset $K_\ep$ of $X$ such that $|\mu_n|(X\setminus K_\ep)<\ep$ for all $n\in\N$.
\end{remark}
\begin{proof}
Let $\ep>0$ and $K_\ep$ a compact set of $X$ such that $|\mu_n|(X\setminus K_\ep)<\ep$ for all $n\in\N$. Since $X$ is $\sigma$-compact and $\mu$ is bounded, we can moreover assume that  $|\mu|(X\setminus K_\ep)<\ep$. Let $g_\ep:X\to[0,1]$ a compactly supported continuous function such that $g_\ep\bigl|_{K_\ep}\equiv1$. Since the function $fg_\ep$ is continuous outside a $\nu$-negligible set and vanishing at infinity, we can apply 
\cite[Lemma 6.3.1]{Chemin} to deduce that
$$
\lim_{n\to\infty}\int fg_\ep\,d\mu_n=\int fg_\ep\,d\mu.
$$
We infer that 
\begin{multline*}
\limsup_{n\to\infty}\Bigl|\int f\,d\mu_n-\int f\,d\mu\Bigr|
=\limsup_{n\to\infty}\Bigl|\int f(1-g_\ep)\,d\mu_n-\int f(1-g_\ep)\,d\mu\Bigr|\\
\leq \|f\|_{L^\infty}(|\mu_n|(X\setminus K_\ep)+|\mu|(X\setminus K_\ep))\leq 2\ep \|f\|_{L^\infty}.
\end{multline*}
Letting $\ep\to0$ completes the proof.
\end{proof}

\section{Velocity, vorticity and circulation: classical setting}

One important dynamic variable for incompressible flow, specially in 2D, is the {\it vorticity}, the curl of velocity. If $u = (u_1,u_2)$ is the velocity then
 the vorticity is $\om=\partial_{x_1}u_2-\partial_{x_2}u_1 \equiv \mbox{ curl } u$. 
In bounded simply connected domains the velocity can be easily recovered from the vorticity by means of the regularizing 
linear operator $\nabla^{\perp}\Delta^{-1}$, where $\Delta^{-1}$ is the inverse Dirichlet Laplacian.  In our case, the fluid domain is neither simply connected nor bounded. 
The purpose of this section is to discuss reconstruction of the velocity from the vorticity for exterior domains, for smooth flows.

 The fact that $\Om$ is not simply connected implies that we need to assign extra conditions in order to recover velocity from vorticity, 
for instance, the circulation of the velocity around each obstacle. For a given vector field $u$, we define $\gamma_j$ the circulation around $\Gamma_j$ as follows 
\begin{equation*}
  \gamma_j=\int_{\Gamma_j}u\cdot\hat\tau\d S.
\end{equation*}

As in the previous section, we denote by $K$ the Biot-Savart operator $\nabla^{\perp}\Delta^{-1}$, where $\Delta$ is the Dirichlet Laplacian in $\Om$, 
and we abuse notation, denoting also by $K = K(x,y)$ its singular kernel. The fact that $\Om$ is unbounded implies that the properties 
of the Biot-Savart operator are a delicate issue. We discuss how to express the velocity field $u$ from the vorticity $\om$ and the circulations $\gamma_1,\dots,\gamma_k$. 
We start by studying the Biot-Savart operator on $\Om$. For $\om\in C^\infty_c (\Om)$ we denote by $K[\om]$ the value of the operator $K$ on $\om$, given by
\begin{equation*}
  K[\om](x)=\int_{\Om}K(x,y)\om(y)\d y.
\end{equation*}
We will also use the analogous notation for the Green's function
\begin{equation*}
  \Delta^{-1} \omega = G[\omega](x)=\int_{\Om}G(x,y)\omega(y)\d y.
\end{equation*}
We have the following

\begin{proposition}\label{propdefK}
Let  $\om\in C^\infty_c (\Om)$. We have that  $K[\om]$ is smooth, divergence free, tangent to the boundary, square integrable on $\Om$ and such that $\curl K[\om]=\om$. Moreover, the circulation of $K[\om]$ on $\Gamma_j$ is given by
\begin{equation}\label{circK}
 \int_{ \Gamma _j}   K [ \omega] \cdot \hat{\tau} \d S = -  \int_{\Omega  }  \w _{j} \omega\d x . 
\end{equation}
\end{proposition}
\begin{proof}
  Denoting $\psi=G[\om]$ 
we have that 
\begin{equation} \label{psi}
\left\{
\begin{array}{ll}
\Delta \psi = \omega, & \text{ in } \Om,\\
 \psi=0, &\mbox{ on } \partial\Om,\\
K[\om]=\nabla^{\perp}\psi, & \text{ in } \overline\Om.
\end{array}
\right.
\end{equation}

We infer that $K[\om]$ is smooth, divergence free, tangent to the boundary and such that $\curl K[\om]=\om$. Moreover, since $\omega$ has compact support,  from \eqref{estk} 
we deduce that $K[\om]$ decays like $\mathcal{O}(1/|x|^2)$ as $|x|\to\infty$ so $K[\om]\in L^2(\Om)$. It remains to prove the formula for the circulation \eqref{circK}.

We have
\begin{align*}
 \int_{ \Gamma _j}   K [ \omega] \cdot \hat{\tau}\d S 
&=\int_{ \partial \Omega }     (K [ \omega])^\perp  \w _{j}\cdot \hat{ n} \d S(\sigma)   \\
&= \int_{\Omega \cap B(0;M)}  \text{div} \   [ (K [ \omega])^\perp\w _{j} ]  \d x
 - \int_{ \partial   B(0;M) }    (K [ \omega])^\perp \w _{j}\cdot\frac x{|x|}\d S(\sigma)\\
&=  \int_{\Omega \cap B(0;M) }  \w _{j}  \text{div} \ [ (K [ \omega])^\perp] \d x+ \int_{\Omega \cap B(0;M) }  
\nabla  \w _{j} \cdot  (K [ \omega])^\perp  \d x \\
&\hskip 6cm - \int_{ \partial   B(0;M) }   (K [ \omega])^\perp \w _{j}\cdot \frac x{|x|} \d S(\sigma) \\
&\equiv I_1 + I_2 + I_3.
\end{align*}

Let us examine each of these three terms. We have, for the first term,
\begin{multline*}I_1 = 
 \int_{\Omega \cap B(0;M) }  \w _{j}  \text{div} \  (K [ \omega])^\perp  \d x = - 
  \int_{\Omega \cap B(0;M) }  \w _{j} \text{ curl }  (K [ \omega])\d x\\
= - \int_{\Omega \cap B(0;M) }  \w _{j}  \omega \d x =  - \int_{\Omega }  \w _{j} \omega\d x,
\end{multline*}
if $M$ is sufficiently large.

Using \eqref{psi} and integrating by parts the second term then gives
\begin{multline*}
  I_2 = \int_{\Omega \cap B(0;M) }  \nabla  \w _{j} \cdot  (K [ \omega])^\perp  \d x
=
  \int_{\Omega \cap B(0;M) } ( \Delta \w _{j}) \,( \psi) \d x \\
- \int_{ \partial   B(0;M) }  \psi \nabla \w _{j}\cdot\frac x{|x|} \d S(\sigma) 
- \int_{ \partial   \Omega  }  \psi \nabla \w _{j}  \cdot\hat{n} \d S(\sigma) 
= - \int_{ \partial   B(0;M) }  \psi \, \nabla \w _{j}\cdot\frac x{|x|} \d S(\sigma).
\end{multline*}
We have used here that $\Delta \w _j = 0$ in 
$\Omega$ and that $\psi= 0$ on $\partial \Omega$, see \eqref{psi}.
Now, the same argument used to show the pointwise estimates for $K[\omega]$ in Proposition \ref{propK}  can be used to show that 
$|\nabla \w _j| = \cO(M^{-2})$ for $|x|=M$. Furthermore, from \eqref{estg} we deduce that $\psi=G[\om]$ is bounded. Thus we obtain
\begin{equation*}|I_2| \to 0\end{equation*}
as $M\to\infty$.

Finally, we note that
\begin{equation*}
I_3 = \int_{ \partial   B(0;M) } \w _{j}  (K [ \omega])^\perp \cdot \frac x{|x|} \d S(\sigma)   \rightarrow  0
\end{equation*}
as $M \to \infty$, since $\w _j$ is bounded and $|K[\omega] |=\cO(M^{-2})$ for $|x|=M$ (which is a consequence of \eqref{estk}).
This establishes \eqref{circK} and completes the proof of the proposition.
\end{proof}

A result analogous to Proposition \ref{propdefK} for bounded domains was explicitly stated and proved in \cite{LopesFilho2007}, and it was implicit in the analysis contained in \cite{FG1997}. 

As an easy consequence of this proposition we show how to recover velocity from vorticity in an exterior domain with $k$ holes. Denote by $C^1_b(\overline{\Omega})$ the set of $C^1$-functions which are bounded on $\overline{\Omega}$.
\begin{proposition} \label{milton}
Let $  \omega \in C^{\infty}_c(\Omega)$ and let  $ \gamma_j \in \real$, $j = 1,\ldots,k$ be given constants.
 Then there exists one and only one divergence free vector field $u  \in  C^1_b  ( \overline{\Omega}  )$, tangent to 
$\partial\Omega$, vanishing at infinity, such that 
$\mbox{curl }u = \omega$   and
\begin{equation*}
  \int_{\Gamma _j}  u \cdot \hat{\tau} \d S = \gamma_j , \;\; j=1,\ldots,k .
\end{equation*}
This vector field is given by the formula:
\begin{equation}\label{recover} 
u =  K [ \omega] + \sum_{j=1}^k \,  \left( \int_{\Omega }  \w _j   \omega   \d x + \gamma_j \right)  \,  \bX_j.
\end{equation}
 \end{proposition}
\begin{proof}
The existence part is trivial since, from \eqref{propsPsij} and Proposition \ref{propdefK}, it follows that the vector field defined in \eqref{recover} has all the required properties.

Now, if 
$v\in C^1_b(\overline{\Omega})$ is another divergence free vector field, tangent to the boundary, vanishing at infinity, whose curl is $\omega$ and whose circulations around each $\Omega_j$ are $\gamma_j$, then consider the difference between 
$v$ and $u$, $\bar{u} \doteq v - u$. We find that $\bar{u}$ will be divergence free, curl-free, tangent to the boundary 
$\partial \Omega$; $\bar{u}$ will vanish at infinity and $\bar{u}$ will have vanishing circulation around each hole 
$\Omega_j$. Thus $\bar{u}$ satisfies the hypothesis of \cite[Lemma 2.14]{Kikuchi83}, which implies that $\bar{u} = 0$. This concludes the proof.
\end{proof}

\section{Velocity, vorticity and circulation: weak setting}


We need to concern ourselves with the technical issue of reconstructing velocity from vorticity and circulations in a less regular setting. This is the subject of the present section.
We will work in the remainder of this paper with velocity fields of the form \eqref{recover} when $\om$ has weak regularity. Given that $K[\om]$ is expected to belong to $L^2$, the natural function space for the velocity field is $L^2(\Om)+\xspace$. Given that all harmonic vector fields $\X_j$ can be written as $\frac{x^\perp}{2\pi|x|^2}+\cO(1/|x|^2)$ when $|x|\to\infty$ we observe that $L^2(\Om)+\xspace=L^2(\Om)+\langle\frac{x^\perp}{|x|^2}\rangle$.

Next, if $u$ can be decomposed as in \eqref{recover} then $\curl u=\curl K[\om]$, so we expect the vorticity to be the curl of a square integrable velocity field. It is then natural to introduce the space 
\begin{equation*}
\H=\{\om\in \mathscr{D}'(\Om);\ \exists v\in L^2(\Om)\text{ such that } \dive v=0,\  \curl v=\om,\ v\cdot\hat{n}\bigl|_{\partial\Om}=0\}.
\end{equation*}
endowed with the norm
\begin{equation}\label{defnormh}
\|\om\|_{\H}=\inf\{\|v\|_{L^2(\Om)}\ ;\ v\in L^2(\Om), \ \dive v=0,\  \curl v=\om,\ v\cdot\hat{n}\bigl|_{\partial\Om}=0\}.
\end{equation}
Clearly $\H$ can be identified with the quotient space $\H=Y/M$, where $Y=\{v\in L^2(\Om), \ \dive v=0,\ v\cdot\hat{n}\bigl|_{\partial\Om}=0\}$ is endowed with the $L^2$ norm, and where $M=\{v\in Y, \  \curl v=0\}$. The space $Y$ is a classical space in mathematical fluid mechanics and is well-known to be a Hilbert space. We have that $M$ is a closed subspace of $Y$. Indeed, the convergence in $L^2$ implies the convergence in the sense of distributions which, in turn, implies the convergence of the curl in the sense of distributions. Therefore the $L^2$ limit of a sequence of curl-free vector fields is a curl-free vector field. We conclude that $\H$ is a Banach space as the quotient of a Banach space by a closed subspace.

\begin{remark}
It is reasonable to ask whether $\H$ is the same as $H^{-1}(\Omega)$.  This is obviously true if the underlying domain is smooth and bounded, but problematic in unbounded domains.
In fact, if we take $\omega$ such that its Fourier transform  $\widehat{\omega}(\xi)$  is given by $\varphi(\xi) (\log |\xi|)^{-1/2}$, where $\varphi$ is a cutoff, identically $1$ for $|\xi| < 1/2$, vanishing for $|\xi| > 2/3$, 
then $\omega \in H^{-1}(\real^2)$ but $\omega \notin \mathscr{H}^{-1}(\real^2)$. Nevertheless, it is not hard to show that for vorticity with compact support in $\Omega$, the two notions  $\H$ and $H^{-1}(\Omega)$ are the same.
\end{remark}

In this paper we are interested in flows with vortex sheet regularity, and, to this end, we will consider velocities in $\L$ and vorticities in $\H\cap\bm$.

Next we show that, for any locally integrable velocity field  whose vorticity is in $\bm$, one can define the circulation on each connected component of the boundary. More precisely, we have the following lemma.
\begin{lemma}\label{defcirc}
Let $v\in L^1\loc(\Om)$ be a vector field such that $\om\equiv\curl v\in \bm $. Then the circulations $\gamma_1,\dots,\gamma_k$ of $v$ on the  connected components of the boundary $\Gamma_1,\dots,\Gamma_k$ are well defined through the following formula:
\begin{equation}\label{defcirceq}
\int_\Omega\varphi\d \om + \int_\Omega v\cdot\nabla^\perp\varphi=-\sum_{j=1}^k\gamma_j\;\varphi\bigl|_{\Gamma _j}
\end{equation}
for all $\varphi\in \yinf$.
\end{lemma}
 
\begin{remark} 
For consistency, we note  that, if $v$ is smooth, then the Stokes formula implies that \eqref{defcirceq} holds true with $\gamma_j$ equal to the circulation of $v$ along $\Gamma _j$ given by the usual formula $\gamma_j := -\int_{\Gamma_{j}} u \cdot \hat{n}^{\perp} dS$. 
\end{remark}

\begin{proof}
It is enough to show existence and uniqueness of a set of constants $\{\gamma_j\}$ for which \eqref{defcirceq} holds. We begin with existence.
 Let us fix $j\in\{1,\dots,k\}$ and consider some 
$\varphi_j\in C^\infty_c(\overline\Omega)$ such that $\varphi_j$ is equal to 1 in a neighborhood of $\Gamma_j$ and vanishes in a neighborhood of $\partial\Omega\setminus\Gamma_j$. We set
\begin{equation*}
\gamma_j\equiv -\int_\Omega\varphi_j\d\om -\int_\Omega v\cdot\nabla^\perp\varphi_j.  
\end{equation*}

For $\varphi$ as in the statement of the lemma, we define $L_j=\varphi\bigl|_{\Gamma _j}$ and $\overline\varphi\equiv\varphi-\sum_{j=1}^kL_j\varphi_j$. Clearly  $\overline\varphi\in C_c ^\infty(\Omega)$, so we have, in the sense of distributions, that 
\begin{equation*}
\int_\Omega\overline\varphi\d \om 
=\langle \om,\overline\varphi\rangle_{\mathscr{D}',\mathscr{D}} 
=\langle \curl v,\overline\varphi\rangle_{\mathscr{D}',\mathscr{D}} 
=-\langle v,\nabla^\perp\overline\varphi\rangle_{\mathscr{D}',\mathscr{D}} 
=-\int_\Omega v\cdot\nabla^\perp\overline\varphi. 
\end{equation*}
We infer that
\begin{equation*}
\int_\Omega\varphi\d \om + \int_\Omega v\cdot\nabla^\perp\varphi
=\sum_{j=1}^k L_j\bigl(\int_\Omega\varphi_j\d \om + \int_\Omega v\cdot\nabla^\perp\varphi_j\bigr) 
= -\sum_{j=1}^k L_j\gamma_j
\end{equation*}
so that \eqref{defcirceq} holds true. This proves existence of the constants $\gamma_1,\dots,\gamma_k$ in \eqref{defcirceq}. Uniqueness 
is clear since, if  $\gamma'_1,\dots,\gamma'_k$ were to verify \eqref{defcirceq} as well, then we would have
\begin{equation*}
\sum_{j=1}^k(\gamma_j-\gamma'_j)\;\varphi\bigl|_{\Gamma _j}
=0  
\end{equation*}
for all $\varphi\in \yinf$. Choosing $\varphi=\varphi_j$ above implies that $\gamma_j=\gamma'_j$. 
\end{proof}
\begin{remark}\label{defcircremark}
If we assume that $v$ is $L^1$ up to the boundary, \textit{i.e.} $v\in L^1_{loc}(\overline\Om)$, then we can allow for the more general class of test functions $\varphi\in\yinfb$ in \eqref{defcirceq}. Let us prove this. Suppose first that $\varphi\in\yinfb$ is vanishing on the boundary and use relation \eqref{defcirceq} with test function $\varphi^k=\varphi\chi_k$ where $\chi_k$ is the cutoff function defined on page \pageref{chin}. Because $\varphi$ vanishes on the boundary, we have that $\varphi^k$ and $\nabla\varphi^k$ are uniformly bounded. By the dominated convergence theorem we infer that $\int_\Om \varphi^k\,d\om\to\int_\Om \varphi\,d\om$ and $\int_\Om v\cdot\nabla^\perp\varphi^k\to \int_\Om v\cdot\nabla^\perp\varphi$ as $k\to\infty$ implying that relation \eqref{defcirceq} holds true for all $\varphi\in\yinfb$ vanishing on the boundary. If $\varphi\in\yinfb$ does not vanish on the boundary, then $\varphi-\sum_{j=1}^k\varphi\bigl|_{\Gamma_j}\varphi_j$ does. Since $\varphi-\sum_{j=1}^k\varphi\bigl|_{\Gamma_j}\varphi_j$ and $\varphi_j$ can be used as test functions in \eqref{defcirceq}, so does $\varphi$.
\end{remark}

Our next objective is to define the Biot-Savart operator for vorticities in $\H\cap\bm$. First we prove that divergence free vector fields in $\L$ which are tangent to the boundary are uniquely determined by 
their curl, together with the circulations around each boundary component. 

\begin{proposition}\label{unicdec}
Let $u\in\L$ be divergence free, curl free, tangent to the boundary and with vanishing circulation on each of the connected components of the boundary. Then $u=0$.  
\end{proposition}
\begin{proof}
We show first that $u\in L^2(\Om)$. We have that $u$ can be decomposed as follows:
\begin{equation*}
u=v+\sum_{j=1}^k\al_j \X_j,  
\end{equation*}
where the vector field $v\in L^2(\Om)$ is divergence free, curl free, tangent to the boundary and has circulation $-\al_j$ on $\Gamma_j$, for all $j$. 

Let $R$ be such that $\Om^c\subset B(0;R)$ and let $\kappa\in C_c(\overline\Om;[0,1])$ be such that $\supp\kappa\subset B(0;2R)$ and $\kappa\equiv1$ on $B(0;R)$. Using \eqref{defcirceq} we have that, for any $n\geq1$,
\begin{multline*}
\left| \sum_{j=1}^k\al_j \right|=\left|\int_\Om v\cdot\nabla^\perp[\kappa(x/n)] \right| 
=\left|\frac1n \int_\Om v\cdot\nabla^\perp\kappa(x/n) \right|
\leq \frac1n \|v\|_{L^2(Rn<|x|<2Rn)}\|\nabla^\perp\kappa(x/n) \|_{L^2}\\
\leq \|v\|_{L^2(Rn<|x|<2Rn)}\|\nabla^\perp\kappa\|_{L^2}
\stackrel{n\to\infty}{\xrightarrow{\hspace*{1cm}}}0.
\end{multline*}
Therefore we have that $\sum_{j=1}^k\al_j=0$. Since, for all $j$, we have
\begin{equation*}
 \X_j=\frac{x^\perp}{2\pi|x|^2}+\cO(1/|x|^2)\in \frac{x^\perp}{2\pi|x|^2}+L^2(\Om), 
\end{equation*}
we infer that 
\begin{equation*}
 \sum_{j=1}^k\al_j \X_j   \in L^2(\Om), 
\end{equation*}
so $u\in L^2(\Om)$.

Recall that the closure of $C^\infty_{c,\sigma}(\Om)$ in $L^2(\Om)$ is the space of square integrable, divergence free vector fields tangent to the boundary. Therefore there exists a sequence $u_j\in C^\infty_{c,\sigma}(\Om)$ such that $u_j\to u$ in $L^2(\Om)$ as $j\to\infty$. Since $u_j\in C^\infty_{c,\sigma}(\Om)$ there exists some $\psi_j\in\yinf$ such that $u_j=\nabla^\perp\psi_j$. We now use relation \eqref{defcirceq} together with the fact that $u$ is curl free and has vanishing circulation on each of the connected components of the boundary, to deduce that
\begin{equation*}
\int_\Om u\cdot u_j=  \int_\Om u\cdot \nabla^\perp\psi_j=0.
\end{equation*}
Letting $j\to\infty$, we infer that $\int_\Om|u|^2=0$, so that $u=0$. This completes the proof.
\end{proof}

Using Proposition \ref{unicdec}, we can define an extension of the Biot-Savart operator $K$ to $\H\cap\bm$.

\begin{definition}\label{defKfield}
Let   $\om\in\H\cap\bm$. We define $K[\om]$ as the unique vector field in $\L$ which is divergence free, tangent to the boundary, with curl equal to $\om$  and such that, for any $j$, its circulation on $\Gamma_j$ is $-\int_\Om\w_j\d\om$. 
\end{definition}

Note that, by definition, if $\omega \in \H$ there exists a vector field $v \in \L$ with $\mbox{curl} v = \omega$, divergence free and tangent to the boundary. If, in addition, $\omega$ is assumed to be 
a bounded measure, then the circulations of $v$ around boundary components are well-defined; this is the content of Lemma \ref{defcirc}. All that is needed to find $K[\omega]$ is to adjust $v$ by adding
a suitable linear combination of harmonic vector fields, so that the resulting field has the required circulations. Uniqueness follows from Proposition \ref{unicdec}.
 
We will need, in the sequel, an approximation result. We introduce some additional notation. Let $\chi_n$ be a cutoff function at distance $\frac1n$ from the boundary and for $|x|<n$. More precisely, we assume that \label{chin}
\begin{gather*}
\chi_n\in C_c^\infty(\Omega;[0,1]), \quad 
\chi_n\equiv1 \text{ in }\Sigma_{\frac2n}^c\cap B_n,
\quad \chi_n\equiv0 \text{ in }\Sigma_{\frac1n}\cap B_{2n}^c\\
 \|\nabla\chi_n\|_{L^\infty(\Sigma_{\frac2n}\setminus\Sigma_{\frac1n})}\leq Cn,\quad  \|\nabla\chi_n\|_{L^\infty(B_{2n}\setminus B_n)}\leq \frac Cn,
\end{gather*}
where
\begin{equation*}
\Sigma_a=\{x\in\Omega\ ;\ d(x,\partial\Omega)\leq a\}\quad\text{and}\quad   B_a=B(0;a).
\end{equation*}

Fix 
\begin{equation*}
\eta\in C^\infty_c(\R^2;\R_+),\quad \supp\eta\subset B_{\frac12},\quad \int\eta=1, 
\end{equation*}
and set 
\begin{equation}\label{etan}
  \eta_n(x)=n^2\eta(nx).
\end{equation}
Before we state the main result of this section we need to prove the following Poincaré inequality.
\begin{lemma}\label{Poincare-lemma}
There exists a constant $C=C(\Omega)$ such that for any $R>0$ and for any function $f\in H_{loc}^1(\overline\Om)$ which vanishes on $\partial\Om$ and such that $\nabla f\in L^2(\Om)$, we have the following inequality:
\begin{equation}\label{Poincare}
\|f\|_{L^2(\Sigma_R)}\leq CR  \|\nabla f\|_{L^2(\Om)}.
\end{equation}
\end{lemma}
\begin{proof}
Relation \eqref{Poincare} with some constant $C(R)$ instead of $CR$ is well-known. We only need to see how the constant depends on $R$. Therefore, we only need to consider the cases $R$ small and $R$ large. We infer that it suffices to assume that the exterior domain $\Om$ have only one hole. Indeed, we can extend $f$ with zero values on all holes except one of them. The resulting extension is still in $H^1_{loc}$ and it suffices to work with this extension.

Suppose now that $\Om=\{x\ ;\ |x|>1\}$. By density, we may assume that $f$ is smooth. Using the polar coordinates $(x,y)=r(\cos\theta,\sin\theta)$ we may write
\begin{equation*}
f(r,\theta)=\int_{1}^{r}\partial_\tau f(\tau,\theta)\,d\tau  
\end{equation*}
so that
\begin{equation*}
|f(r,\theta)|^2\leq  (r-1)\int_{1}^{r}|\partial_\tau f(\tau,\theta)|^2\,d\tau   
\end{equation*}
Then
\begin{align*}
 \|f\|_{L^2(\Sigma_R)}^2
&=\int_{1<|x|<R+1}|f(x)|^2\,dx\\ 
&=\int_{1}^{1+R}\int_0^{2\pi}|f(r,\theta)|^2\,dr\,d\theta\\
&\leq \iint_{1<\tau<r<1+R}\int_0^{2\pi}(r-1)|\partial_\tau f(\tau,\theta)|^2\,d\tau   \,dr\,d\theta\\
&= \int_{1}^{1+R}\int_0^{2\pi}\bigl[\frac{R^2}2-\frac{(\tau-1)^2}2\bigr]|\partial_\tau f(\tau,\theta)|^2\,d\tau  \,d\theta\\
&\leq \frac{R^2}2  \|\partial_r f\|_{L^2(\Sigma_R)}^2.
\end{align*}

Suppose now that $\R^2\setminus\Om$ is not a disk. By the Kellogg-Warschawski theorem (see for example \cite[Theorem 3.6]{Pom}) there exists a biholomorphism $T:\{|x|>1\}\to\Omega$ which is smooth up to the boundary. In fact, \cite[Theorem 3.6]{Pom} is stated for simply connected domains, but using the inversion $z\mapsto1/z$ it holds true for the exterior of a simply connected obstacle too. Moreover, there exist $C_1,C_2>0$ such that $C_1<|T'(z)|<C_2$ for all $z$. In particular, we have that $\Sigma_R\subset T(\Sigma_{C_3R}(D^c))$ for some constant $C_3>0$ (here $\Sigma_{R}(D^c)$ denotes the set $\Sigma_R$ associated to $D^c=\{|x|>1\}$). Therefore $\|f\|_{L^2(\Sigma_R)}\leq C_2\|f\circ T\|_{L^2(\Sigma_{C_3R}(D^c))}$. Moreover, we also have that $\|\nabla (f\circ T)\|_{L^2(|x|>1)}\leq C_4 \|\nabla f\|_{L^2(\Om)}$ for some constant $C_4>0$, so relation \eqref{Poincare} in the general case follows from the particular case of the exterior of the unit disk applied to the function $f\circ T$.
\end{proof}

We now prove the following approximation result.
\begin{proposition}\label{aproxom}
Let $\om\in\H\cap\bm$. We have that $K[\om]\in L^2(\Om)$. Moreover, let us define $  \omega^n=(\chi_n\omega)\ast\eta_n$. Then we have that 
\begin{gather}
\omega^n\text{ is bounded in } L^1(\Omega),\label{conv2}\\
\int_\Omega \varphi\omega^n\stackrel{n\to\infty}\longrightarrow \int_\Omega\varphi\d\omega  \text{ and }  \quad\forall\varphi\in C^0_b(\overline\Omega),\label{conv1}\\
K[\om^n] \stackrel{n\to\infty}\longrightarrow K[\om] \text{ strongly in } L^2(\Om)\label{conv123}.
\end{gather}
Moreover, any weak limit in the sense of measures of any subsequence of $|\om^n|$ is a continuous measure.
\end{proposition}
Above, we denoted by $C^0_b(\overline\Omega)$ the space of bounded functions on $\Omega$, continuous up to the boundary and with a finite limit as $|x|\to\infty$.

\begin{proof}
Claim \eqref{conv2} is obvious:
\begin{equation*}
\int_\Omega|\omega^n|=\int _\Omega | (\chi_n\omega)\ast\eta_n|
\leq \int_\Omega \chi_n \d|\omega| \leq \int_\Omega1\d|\omega|=|\omega|(\Omega)<\infty.
\end{equation*}

Next we prove \eqref{conv1}. Let $\varphi\in C^0_b(\overline\Omega)$. We extend it to a continuous function on $\R^2$, again denoted by $\varphi$.  For $f$ smooth, denote $\check f = \check f (z) = f(-z)$. Then we have:  
\begin{multline*}
 \int_\Omega \varphi\omega^n- \int_\Omega\varphi\d\omega 
=\int_\Omega \varphi(\chi_n\omega)\ast\eta_n- \int_\Omega\varphi\d\omega
=\int_\Omega\chi_n(\varphi\ast\check\eta_n)\d\omega - \int_\Omega\varphi\d\omega\\
=\int_\Omega\chi_n(\varphi\ast\check\eta_n-\varphi)\d\omega +\int_\Omega(\chi_n-1)\varphi\d\omega\equiv I_1+I_2.
\end{multline*}

We first bound $I_2$:
\begin{equation*}
|I_2|\leq \int_{\Sigma_{\frac2n}\cup B_n^c}  |(\chi_n-1)\varphi|\d|\omega|\leq \|\varphi\|_{L^\infty}(|\omega|(\Sigma_{\frac2n})+|\omega|(B_n^c))\stackrel{n\to\infty}\longrightarrow 0.
\end{equation*}

Next, we know by classical results that $\varphi\ast\check\eta_n\to\varphi$ uniformly in $\R^2$. Therefore
\begin{equation*}
|I_1|\leq   \|\varphi\ast\check\eta_n-\varphi\|_{L^\infty(\R^2)}|\omega|(\Omega)\stackrel{n\to\infty}\longrightarrow 0.
\end{equation*}
This completes the proof of \eqref{conv1}.

\bigskip

We prove now that any weak limit in the sense of measures of any subsequence of $|\om^n|$ is a continuous measure. Let $\om_+$, respectively $\om_-$, be the positive part, respectively the negative part, of the measure $\om$. Since $\om\in\H\subset H^{-1}(\Omega)$ we infer that $\om$ is a continuous measure (see \cite[Lemma 6.3.2]{Chemin}). Since $\om=\om_++\om_-$ and $\om_+$, $\om_-$ are orthogonal measures, we have that $\om_+$ and $\om_-$ are also continuous measures. Let $\om^n_+=(\chi_n\om_+)\ast\eta_n$ and  $\om^n_-=(\chi_n\om_-)\ast\eta_n$. Then $\om^n_+$ and $\om_-^n$ are single-signed, although they are not necessarily the positive and negative parts of $\om^n$. Nevertheless, we have the bound $|\om^n|\leq \om^n_+-\om^n_-$. We can apply the result proved in relation \eqref{conv1} to the measures $\om_+$ and $\om_-$. Indeed, we did not use the hypothesis that $\om\in\H$ to prove that result. We therefore have that $\om^n_+\rightharpoonup\om_+$ and $\om^n_-\rightharpoonup\om_-$ weakly in the sense of measures. Therefore $\om^n_+-\om^n_-\rightharpoonup\om_+-\om_-=|\om|$ weakly in the sense of measures. From the bound $|\om^n|\leq \om^n_+-\om^n_-$ we infer that any weak limit in the sense of measures of any subsequence of $|\om^n|$ must be bounded by $|\om|$. Since $|\om|$ is a continuous measure, we deduce that any weak limit in the sense of measures of any subsequence of $|\om^n|$ must be a continuous measure.

\bigskip

The proof that $K[\om]\in L^2(\Omega)$ and the proof of \eqref{conv123} is done in three steps. First we show that  $K[\om^n]$ is bounded in $L^2(\Om)$. Second, we prove that $K[\om^n]$ converges weakly to $K[\om]$ in $L^2(\Omega)$ (showing in particular that $K[\om]\in L^2(\Omega)$). Third, we show the strong convergence of $K[\om^n]$ to $K[\om]$ in $L^2(\Omega)$.

\medskip

We prove now that $K[\om^n]$ is bounded in $L^2(\Om)$. Let $f\in C^\infty_c(\Om)$ be a divergence free vector field and let us define $F=G[\curl f]$. Clearly $F\in C^\infty(\overline\Om)$ is bounded and vanishes on $\partial\Om$. We extend it to $\R^2$ by setting $F$ to vanish on $\Om^c$.
We have that
\begin{multline}\label{idistr}
\int_\Om K[\om^n]\cdot f =\int_\Om\nabla^\perp G[\om^n]\cdot f
= -\int_\Om G[\om^n] \curl f
=- \int_\Om G[\om^n]\triangle F
=- \lim_{R\to\infty}\int_{\Om\cap B_R} G[\om^n]\triangle F\\
=- \lim_{R\to\infty}\Bigl(\int_{\Om\cap B_R} \triangle G[\om^n]  F+\int_{|x|=R} G[\om^n]\ \frac x{|x|}\cdot \nabla G[\curl f]-\int_{|x|=R} \frac x{|x|}\cdot\nabla G[\om^n]\cdot  G[\curl f]\Bigr)
=-\int_\Om \om^nF.
\end{multline}
We used above that, for a compactly supported smooth function $h$, we have that $G[h](x)=\cO(1)$ and $\nabla G[h](x)=\cO(1/|x|^2)$, as $|x|\to\infty$. We recall now that $\om\in\H$ and we choose some $v\in L^2(\Om)$ such that $\dive v=0$, $\curl v=\om$ and $v$ is tangent to the boundary. Next, we observe that
\begin{multline}\label{jdistr}
\int_\Om \om^nF=\int_\Om (\chi_n\om)\ast\eta_n\ F
=\int_\Om \chi_n(F\ast\check\eta_n)\d\om
=\langle\om,  \chi_n(F\ast\check\eta_n)\rangle_{\mathcal{D}'(\Om),\mathcal{D}(\Om)}\\
=\langle\curl v,  \chi_n(F\ast\check\eta_n)\rangle_{\mathcal{D}'(\Om),\mathcal{D}(\Om)}
=-\langle v ,  \nabla^\perp[\chi_n(F\ast\check\eta_n)]\rangle_{\mathcal{D}'(\Om),\mathcal{D}(\Om)}\\
=-\int_\Om v \nabla^\perp\chi_n(F\ast\check\eta_n)
-\int_\Om v \chi_n(\nabla^\perp F\ast\check\eta_n).
\end{multline}
We estimate each of the two last terms above:
\begin{equation*}
 \bigl|\int_\Om v \chi_n(\nabla^\perp F\ast\check\eta_n)\bigr|\leq \|v \|_{L^2(\Om)}\|\nabla^\perp F\ast\check\eta_n\|_{L^2(\Sigma_{\frac1n}^c\cap B_{2n})}\leq \|v \|_{L^2(\Om)}\|\nabla^\perp F\|_{L^2(\Om)}
\end{equation*}
and
\begin{align*}
\bigl| \int_\Om v \nabla^\perp\chi_n(F\ast\check\eta_n) \bigl|
&=\bigl| \int_{(\Sigma_{\frac2n}\setminus\Sigma_{\frac1n})\cup(B_{2n}\setminus B_n)} v \nabla^\perp\chi_n(F\ast\check\eta_n) \bigl|\\
&\leq \|v \|_{L^2(\Om)}\bigl(Cn\|F\|_{L^2(\Sigma_{\frac5{2n}}\setminus\Sigma_{\frac1{2n}})}+\frac Cn\|F\|_{L^2(B_{2n+1}\setminus B_{n-1})}\bigr)\\
&\leq C\|v \|_{L^2(\Om)}\|\nabla F\|_{L^2(\Om)}
\end{align*}
where we used Lemma \ref{Poincare-lemma} and the fact that $F$ vanishes on $\partial\Om$.

We infer from \eqref{idistr} and \eqref{jdistr}, together with the estimates performed above, that the following inequality holds true:
\begin{equation}\label{conv8}
 \bigl| \int_\Om K[\om^n]\cdot f\bigr|\leq C  \|v \|_{L^2(\Om)}\|\nabla F\|_{L^2(\Om)}.
\end{equation}

Next we observe that
\begin{equation}\label{conv7}
 \|\nabla F\|_{L^2(\Om)}\leq \|f\|_{L^2(\Om)}. 
\end{equation}
Indeed, we have that $\triangle F=\curl f$ on $\Om$ so
\begin{equation} \label{Fstuff1}
-\int_\Om\triangle F\ F=  -\int_\Om\curl f\ F=\int_\Om f\cdot\nabla^\perp F\leq \|f\|_{L^2(\Om)}\|\nabla F\|_{L^2(\Om)}.
\end{equation}
On the other hand,
\begin{equation} \label{Fstuff2}
 -\int_\Om\triangle F\ F
=\int_\Om|\nabla F|^2-\lim_{R\to\infty}\int_{|x|=R}\frac x{|x|}\cdot\nabla F\ F= \|\nabla F\|_{L^2(\Om)}^2.
\end{equation}
Putting together \eqref{Fstuff1} and \eqref{Fstuff2}, we get \eqref{conv7}. Using \eqref{conv7} in \eqref{conv8} we obtain that
\begin{equation*}
 \left| \int_\Om K[\om^n]\cdot f\right|\leq C  \|v \|_{L^2(\Om)}\|f\|_{L^2(\Om)} , 
\end{equation*}
for all divergence free vector fields $f\in C^\infty_c(\Om)$. This implies that
\begin{equation*}
 \|K[\om^n]\|_{L^2(\Om)} \leq C  \|v \|_{L^2(\Om)}.
\end{equation*}

\medskip

We have, hence, established that  $K[\om^n]$ is bounded in $L^2$. Therefore, there exists a subsequence $K[\om^{n_k}]$ which converges weakly to some vector field $w\in L^2(\Om)$. Given that  $K[\om^{n_k}]$ is divergence free and tangent to the boundary, the same holds true for the weak limit $w$. Moreover, we have that $K[\om^{n_k}]\to w$ in $\mathscr{D}'(\Om)$ so $\om^{n_k}=\curl K[\om^{n_k}]\to \curl w$ in $\mathscr{D}'(\Om)$. Using  \eqref{conv1} we have that $ \om^{n_k}\to\om$ in $\mathscr{D}'(\Om)$, hence $\curl w=\om$. We use Proposition \ref{propdefK} and Lemma \ref{defcirc} with $\om$ replaced by $\om^{n_k}$ and $w$ replaced by $K[\om^{n_k}]$ to deduce the following identity:
\begin{equation*}
\int_\Om\varphi   \om^{n_k}+\int_\Om K[\om^{n_k}]\cdot\nabla^\perp\varphi=\sum_{j=1}^k\varphi\bigl|_{\Gamma_j}\int_\Om\w_j  \om^{n_k},
\end{equation*}
for all $\varphi\in \yinf$. Letting $n_k\to\infty$ and using \eqref{conv1} together with the fact that $K[\om^{n_k}]\rightharpoonup w$ weakly in $L^2$ implies that 
\begin{equation*}
\int_\Om\varphi   \d\om+\int_\Om w\cdot\nabla^\perp\varphi=\sum_{j=1}^k\varphi\bigl|_{\Gamma_j}\int_\Om\w_j  \d\om.
\end{equation*}
From Lemma \ref{defcirc} we infer that $w$ and $K[\om]$ both have  circulation around $\Gamma_j$ equal to $-\int_\Om\w_j\d\om$, so $K[\om]=w$. In particular, we have that $K[\om]\in L^2(\Om)$. We proved that the limit of any weakly convergent subsequence of $K[\om^n]$ in $L^2$ must necessarily be $K[\om]$. We conclude that the whole sequence $K[\om^n]$ converges to $K[\om]$ weakly in $L^2(\Om)$.

\medskip

Finally, we show the strong convergence of $K[\om^n]$ to  $K[\om]$ in $L^2(\Om)$. Given that we already know the weak convergence, it suffices to show the convergence of the norms. We use relations \eqref{idistr} and \eqref{jdistr} with $f=f_n=K[\om^n]$, $F=F_n=G[\curl f_n]=G[\om^n]$ and $v=K[\om]$ to obtain that
\begin{equation*}
\int_\Omega|K[\om^n]|^2=  \int_\Omega K[\om^n]\cdot f_n
=\int_\Omega  K[\om]\nabla^\perp\chi_n(F_n\ast\check{\eta}_n)
+\int_\Omega K[\om]\chi_n(\nabla^\perp F_n \ast\check\eta_n)
\equiv J_1+J_2.
\end{equation*}
In fact, we assumed in  relations \eqref{idistr} and \eqref{jdistr} that $f\in C_c^\infty(\Omega)$ while $K[\om^n]$ is not compactly supported. However, when we look at the proof of these relations, we observe that we only require sufficient decay at infinity for $f$. Since $K[\om^n]=\cO(|x|^{-2})$ as $|x|\to\infty$ is sufficient decay, we see that the proofs of \eqref{idistr} and \eqref{jdistr} go through for the choice $f=K[\om^n]$.

Next, since $f_n=\nabla^\perp F_n$, we have that
\begin{equation*}
J_2=\int_\Omega \eta_n\ast(\chi_n K[\om])\cdot f_n.  
\end{equation*}
By classical results, we know that $\eta_n\ast(\chi_n K[\om])\to K[\om]$ strongly in $L^2(\Omega)$. On the other hand, we have that $f_n=K[\om^n]\to K[\om]$ weakly in $L^2(\Omega)$. Therefore, we can pass to the limit in the term $J_2$:
\begin{equation*}
J_2  \stackrel{n\to\infty}\longrightarrow \int_\Omega|K[\om]|^2.
\end{equation*}

To bound the term $J_1$, recall that $\supp\nabla\chi_n\subset \Sigma_{\frac2n}\cup B_n^c$. We have that
\begin{equation*}
|J_1|\leq \|K[\om]\|_{L^2(  \Sigma_{\frac2n}\cup B_n^c)}\|\nabla^\perp\chi_n(F_n\ast\check{\eta}_n)\|_{L^2(\Omega)}.
\end{equation*}
As in the estimates that follow relation \eqref{jdistr}, we can bound
\begin{equation*}
 \|\nabla^\perp\chi_n(F_n\ast\check{\eta}_n)\|_{L^2(\Omega)}\leq C \|\nabla F_n\|_{L^2(\Omega)}=C \|f_n\|_{L^2(\Omega)}=C \|K[\om^n]\|_{L^2(\Omega)}\leq C'
\end{equation*}
independently of $n$. But clearly
\begin{equation*}
\|K[\om]\|_{L^2(  \Sigma_{\frac2n}\cup B_n^c)}  \stackrel{n\to\infty}\longrightarrow 0
\end{equation*}
so
\begin{equation*}
J_1  \stackrel{n\to\infty}\longrightarrow 0.
\end{equation*}
We conclude from the above relations $\|K[\om^n]\|_{L^2(\Omega)}\to \|K[\om]\|_{L^2(\Omega)}$ as $n\to\infty$. This completes the proof of Proposition \ref{aproxom}.
\end{proof}

\begin{remark}

We proved in relation \eqref{idistr} the following identity:
\begin{equation*}
\int_\Om K[\om]\cdot f=-\int_\Om \om \, G[\curl f]  
\end{equation*}
for all $\om \in C^\infty_c(\Om)$ and $f\in \bigl(C^\infty_{c}(\Om)\bigr)^2$. This would allow us to define $K[\om]$ in the sense of distributions  for all measures $\om\in \bm$, without requiring
$\omega \in \H$. For the sake of convenience we chose, in Definition \ref{defKfield}, to define $K[\om]$ in the subset of measures in $\H$.
\end{remark}

Given a vector field $u\in\L$ which is divergence free and tangent to the boundary such that $\omega=\curl u\in\bm $, one can compute the circulations $\gamma_1,\dots,\gamma_k$ on each of the connected components of the boundary $\Gamma_1,\dots,\Gamma_k$ by using Lemma \ref{defcirc}. Conversely, given vorticity and circulations, one can recover the velocity in this weak setting. The next proposition is the analogue of Proposition \ref{milton} in the weak setting.
\begin{proposition}\label{uomega}
Let $\omega\in\bm\cap \H $ and consider some arbitrary real numbers $ \gamma_1,\dots,\gamma_k$. There exists a unique vector field  $u\in \L$ such that $\curl u=\omega$, $\dive u=0$, $u$ is tangent to the boundary   and with circulations given by $ \gamma_1,\dots,\gamma_k$. Moreover, we have that
\begin{equation}\label{defucirc}
u=K[\omega]+\sum_{j=1}^k\bigl(\gamma_j+\int_\Omega\w _j\d\om \bigr)  \bX_j.
\end{equation}
\end{proposition}
\begin{proof}
The uniqueness part is proved in Proposition \ref{unicdec}. To prove the existence part, we observe that the vector field defined in \eqref{defucirc} has all the required properties. Indeed, since $K[\om]\in L^2(\Om)$ we have that $u\in\L$. It is obvious that $\curl u=\om$, $\dive u=0$ and that $u$ is tangent to the boundary. We conclude the proof by recalling that the circulation of $K[\om]$ around $\Gamma_j$ is given by $-\int_\Omega\w _j\d\omega$. 
\end{proof}

We end this section with the following characterization of the $\H$ norm.
\begin{proposition} \label{charH-1norm}
For any  $\omega\in\bm\cap \H $, we have that $\|\om\|_{\H}=\|K[\om]\|_{L^2(\Om)}.$  
\end{proposition}
\begin{proof}
Given that $K[\om]$ is square-integrable, divergence free, tangent to the boundary and with curl $\om$, we see from the definition of $\|\om\|_{\H}$ given in \eqref{defnormh} that it suffices to show that $  \|K[\om]\|_{L^2(\Om)}\leq \|v\|_{L^2(\Om)}$ for every square integrable vector field $v$ which is divergence free, tangent to the boundary and of curl $\om$. Let $v$ be such a vector field. Then $v-K[\om]$ is square integrable, divergence and curl free and tangent to the boundary. Therefore it must be a linear combination of the harmonic vector fields $X_1,\dots,X_n$:
\begin{equation*}
 v-K[\om]=\sum_{j=1}^k\al_j X_j. 
\end{equation*}
From \eqref{harumph1} we see that
\begin{equation*}
 \sum_{j=1}^k\al_j X_j= \left(\sum_{j=1}^k\al_j \right)\frac{x^\perp}{2\pi|x|^2}+\cO(|x|^{-2})
\end{equation*}
at infinity. Since the left-hand side is square integrable, we deduce that $\sum_{j=1}^k\al_j =0$. 

Clearly
\begin{equation*}
\|v\|_{L^2(\Om)}^2= \|K[\om]\|_{L^2(\Om)}^2+\|\sum_{j=1}^k\al_j X_j\|_{L^2(\Om)}^2+2\int_\Om K[\om]\cdot \sum_{j=1}^k\al_j X_j.
\end{equation*}
It suffices to show that the last term on the right-hand side vanishes. Let $\om^n$ be defined as in Proposition \ref{aproxom}. Then, since $K[\om^n]\to K[\om]$ in $L^2(\Om)$, we have that
\begin{equation*}
 \int_\Om K[\om]\cdot \sum_{j=1}^k\al_j X_j=\lim_{n\to\infty}\int_\Om K[\om^n]\cdot \sum_{j=1}^k\al_j X_j. 
\end{equation*}
But using that $G[\om^n]$ is bounded and that $\sum_{j=1}^k\al_j X_j$ is curl free and $\cO(|x|^{-2})$ at infinity, we have that
\begin{equation*}
\int_\Om K[\om^n]\cdot \sum_{j=1}^k\al_j X_j
=\int_\Om \nabla^\perp G[\om^n]\cdot \sum_{j=1}^k\al_j X_j
=-\int_\Om G[\om^n]\curl \sum_{j=1}^k\al_j X_j
=0
\end{equation*}
since the harmonic fields $X_j$ are curl free. This completes the proof.
\end{proof}

\section{An adaptation of Delort's theorem to exterior domains}
In this section we provide a precise statement and a proof of Delort's Theorem in the exterior domain $\Om$. We consider incompressible ideal 
fluid flow in the exterior domain $\Om \equiv \real^2 \setminus \Omega$, modeled by the Euler equations. Recall the initial-boundary-value problem for the $2$-D incompressible Euler equations:
\begin{equation} \label{2DincEuler}
\left\{
\begin{array}{ll}
\partial_t u + u \cdot \nabla u + \nabla p = 0, &  \mbox{ in } \real_+ \times \Om,\\
\text{div }  u =0, &   \mbox{ in }\overline{\real}_+\times\Om,\\
u \cdot \hat{n} = 0,  &  \text{ on }   \overline{\real}_+ \times \partial \Om ,\\
| u(t,x)|   \rightarrow 0, &  \text{ as }   | x|  \rightarrow \infty, \;\; t\in\real_+,\\
u(0,x)=u_0(x), &  \text{ on } \{t=0\}\times\Om.
\end{array}
\right.
\end{equation}
Here, $u=(u_1,u_2)$ is the velocity field and $p$ is the scalar pressure. The evolution equation for $\omega = \mbox{ curl } u$ is known as the vorticity equation:
\begin{equation*} 
\left\{
\begin{array}{ll}
\partial_t \omega + u \cdot \nabla \omega = 0, &\mbox{ in } \real_+ \times \Omega\\
\mbox{div } u = 0, \;\; \mbox{ curl } u = \omega, &  \mbox{ in } \overline{\R}_+\times\Omega\\
u \cdot \hat{n} = 0,  &  \text{ on }   \overline{\R}_+\times \partial \Omega ,\\
| u(t,x)|   \rightarrow 0, &  \text{ as }   | x|  \rightarrow \infty, \;\; t\in\real_+,\\
\omega(0,x)=\omega_0(x) = \mbox{ curl } u_0, &  \text{ on } \{t=0\}\times\Omega.
\end{array}
\right.
\end{equation*}

We begin with a precise formulation of what it means to be a weak solution of \eqref{2DincEuler}.

\begin{definition} \label{eulerweaksol}
Let $u_0 \in L^2\loc (\overline\Om)$ and $u = u(x,t) \in L^{\infty}\loc([0,\infty); L^2\loc (\overline\Om)) \cap C^0([0,\infty); \mathscr{D}'(\Om))$ 
be a vector field. We will say that $u$ is a {\it weak solution} of the 2D incompressible Euler equations in $\Om$ with initial
data $u_0$ if, for any divergence free 
test vector field $\phi \in C^{\infty}_c([0,\infty)\times\Om)$ we have:
\begin{enumerate}
\item We have the integral identity
\begin{equation} \label{defveleq}
\int_0^\infty \int_{\Om}\partial_t \phi \cdot u  \d x \d t + \int_0^\infty \int_{\Om} \nabla\phi : u \otimes u  \d x\d t  + \int_{\Om} u_0 \cdot \phi(0,\cdot) \d x = 0.
\end{equation} 
\item For each time $t \in [0,\infty)$, the vector field $u(\cdot,t)$ is divergence free in the sense of distributions in $\Om$.
\item The boundary condition $u(\cdot,t) \cdot n=0$ is satisfied in the trace sense in $\partial \Om$, for every time.
\end{enumerate}
\end{definition}

\begin{remark}
The condition that $u\in C^0([0,\infty); \mathscr{D}'(\Om))$ is a consequence of $u  \in L^{\infty}\loc([0,\infty); L^2\loc (\overline\Om))$ and of the integral relation in (a). 
We chose to require it explicitly in Definition \ref{eulerweaksol} only to make sense of parts (b) and (c). Moreover, the fact that $u$ is $L^{\infty}\loc([0,\infty);L^2\loc )$ and $C([0,\infty);\mathscr{D}')$ implies 
that $u$ is continuous into weak $L^2\loc $, which, in particular, implies that it belongs to $L^2\loc $ pointwise in time.
This, together with the (pointwise in time) divergence free condition implies that the normal component of $u$ 
has a trace at the boundary for each fixed time, so that the boundary data can be imposed pointwise in time.     
\end{remark}

In \cite{Delort91}, Delort proved his existence result for bounded domains, full space and compact manifolds without boundaries.
However, his proof of the bounded domain case is {\it local}, because it is based on the local behavior of certain quadratic
expressions in the components of velocity. To extend Delort's result to an exterior domain we first require a strategy to obtain
an approximate solution sequence. This can be done by mollifying initial data and using Kikuchi's existence result, see 
\cite{Kikuchi83}. We give now a precise statement of a version of Delort's existence result for flows in an exterior domain. 
\begin{theorem}\label{theodelort}
Assume that $u_0\in \L $ is divergence free, tangent to the boundary and that $\om_0 = \mbox{ curl } u_0 \in\bmp + L^1(\Om)$. Then there exists a global weak solution $u\in L^\infty([0,\infty);\L)$ of the incompressible Euler equations with initial velocity $u_0$.   
\end{theorem}
\begin{proof}
We start by observing that $\om_0\in\H$. Indeed, we write $u_0 = v_0 + Y$, with $v_0 \in L^2(\Omega)$ and $Y \in \mathscr{X}$. Since both $u_0$ and $Y$ are divergence free and tangent to 
$\partial \Omega$, so is $v_0$ and, since $\mbox{ curl } Y = 0$,  it follows that $\om_0 = \mbox{ curl } v_0$. Hence, $\om_0$  satisfies the requirements in the definition of $\H$.  

By Lemma \ref{defcirc}, the circulations of $u_0$ on each of $\Gamma_1,\dots,\Gamma_k$ are well defined: $\gamma_{1,0},\dots,\gamma_{k,0}$. From Proposition \ref{uomega} we know that
\begin{equation}\label{defuzero}
 u_0=K[\omega_0]+\sum_{j=1}^k\bigl(\gamma_{j,0}+\int_\Omega\w _j\d\om_0 \bigr)  \bX_j=K[\omega_0]+\Y,\quad\text{where } \Y=\sum_{j=1}^k\bigl(\gamma_{j,0}+\int_\Omega\w _j\d\om_0 \bigr)  \bX_j.
\end{equation}
Next, we smooth out $\om_0$ by convolution with $\eta_n$ (defined in \eqref{etan}). We smooth out the initial vorticity as in Proposition \ref{aproxom} by writing $\om_0^n=\eta_n\ast(\chi_n\om_0)$ and we define 
\begin{equation}\label{defuzerobis}
u^n_0=K[\omega^n_0]+\Y.
\end{equation}

From \eqref{conv123} we have that $u_0-u^n_0$ goes to 0 strongly in $L^2(\Om)$. Since $u_0^n$ is smooth, one 
can use Kikuchi's result \cite{Kikuchi83} to construct a global smooth solution $u^n$ of the incompressible Euler equations
 in $\Om$ with initial velocity $u_0^n$. 
Let us define
\begin{equation*}
\om^{n}=\curl u^{n}\quad\text{and}\quad v^{n}=K[\om^{n}].  
\end{equation*}
From Proposition \ref{milton}, we know that the circulation of $u^n_0$ on $\Gamma_j$ is given by 
\begin{equation*}
\gamma_{j,0}+\int_\Om \w _j\d\om_0-\int_\Om \w _j\om_0^n . 
\end{equation*}
Since $u^n$ is smooth, the Kelvin circulation theorem applies so the circulation of $u^n$ on $\Gamma_j$ is given by the same quantity. Using again Proposition \ref{milton} we obtain that
\begin{equation}\label{unprop}
u^n=K[\omega^n]+\sum_{j=1}^k\bigl(\gamma_{j,0}+\int_\Om \w _j\d\om_0-\int_\Om \w _j\om_0^n+\int_\Omega\w _j\om^n \bigr)  \bX_j.
\end{equation}

Observe now that $u^n-\Y=\mathcal{O}(|x|^{-2})$ as $|x|\to\infty$. Indeed, we know that $K[\omega^n]=\mathcal{O}(|x|^{-2})$. Using \eqref{harumph1} we have that
\begin{equation*}
u^n-\Y=\frac{x^\perp}{2\pi|x|^2}\sum_{j=1}^k\Bigl(  \int_\Omega\w _j\om^n -\int_\Om \w _j\om_0^n\Bigr)+\mathcal{O}(|x|^{-2})
=\frac{x^\perp}{2\pi|x|^2}\Bigl(  \int_\Omega\om^n -\int_\Om \om_0^n\Bigr)+\mathcal{O}(|x|^{-2})=\mathcal{O}(|x|^{-2})
\end{equation*}
as $|x|\to\infty$. We used above that $\sum_{j=1}^k\w _j$ is a harmonic function, equal to 1 on $\partial\Omega$ and with a finite limit at $\infty$ so, by uniqueness, it must be identically equal to 1. Similarly $\partial_t(u^n-\Y)=\mathcal{O}(|x|^{-2})$ as $|x|\to\infty$.

Since $\Y$ is time independent, we have that
\begin{equation*}
\partial_t(u^n-\Y)+u^n\cdot\nabla u^n=-\nabla p_n.  
\end{equation*}
In the relation above, the first term on the left-hand side is $\mathcal{O}(|x|^{-2})$ as $|x|\to\infty$. Thanks to \eqref{unprop} and to the fact that $\nabla u^n$ is uniformly bounded we also have that the second term on the left-hand side is $\mathcal{O}(|x|^{-1})$ as $|x|\to\infty$. We conclude that the left-hand side above is $\mathcal{O}(|x|^{-1})$ as $|x|\to\infty$, so it belongs to $L^p(\Omega)$ for all $p>2$. Since the Leray projector $\mathbb{P}$ is well-defined and continuous on every  $L^p(\Omega)$, $1<p<\infty$ (see for example \cite{SS92}), we can apply it to the relation above to obtain that
\begin{equation}\label{equn}
\partial_t(u^n-\Y)+\mathbb{P}(u^n\cdot\nabla u^n)=0.  
\end{equation}
Since the term above belongs to any $L^p(\Omega)$, $p>2$, and since $u^n-\Y\in L^p(\Omega)$ for all $p>1$, we can multiply \eqref{equn} by $u^n-\Y$ to obtain, after integrating by parts, that
\begin{align*}
\frac12\dt\| u^n-\Y\|_{L^2(\Om)}^2
&=-\int_\Om u^n\cdot\nabla u^n\cdot (u^n-\Y)
= -\int_\Om u^n\cdot\nabla Y\cdot (u^n-\Y)\\
&= -\int_\Om (u^n-\Y)\cdot\nabla \Y\cdot (u^n-\Y)
-\int_\Om \Y\cdot\nabla \Y\cdot (u^n-\Y)\\
&\leq \| u^n-\Y\|_{L^2(\Om)}^2 \| \nabla\Y\|_{L^\infty(\Om)}+\| u^n-\Y\|_{L^2(\Om)}\| \Y\cdot\nabla\Y\|_{L^\infty(\Om)}.
\end{align*}
The Gronwall lemma implies that $u^n-Y$ is uniformly bounded in $L^\infty\loc([0,\infty);L^2(\Om))$. From the vorticity equation we have that $\om^n$ is bounded in $L^\infty(\R_+;L^1(\Om))$. Therefore, there exists a subsequence $u^{n_p}$ of $u^n$ and some $u\in L^\infty\loc([0,\infty);\Y+L^2(\Om))$  such that $u^{n_p}-u$ converges to 0 weak$\ast$ in $L^\infty\loc([0,\infty);L^2(\Om))$. Passing to further subsequences as needed, we also have that there exists $\om\in L^\infty(\R_+;\bmp + L^1(\Om))$ so that $\om^{n_p}\to\om$ weak$\ast$ in  $L^\infty(\R_+;\bmp + L^1(\Om))$. To see this, we split $\om_0$ into its positive and negative parts, so that $\om_0^{+} \in \bmp$ and $\om_0^{-}\in L^1(\Om)$. Then, $\om_0^{n,-} = \eta_n \ast (\chi_n\om_0^-)$ converges weakly (even strongly) in $L^1(\Om)$ to $\om_0^-$. Let $\Phi^n=\Phi^n(t,x)$ denote the flow map associated to $u^n$.  Write the corresponding smooth $\om^n$ as $\om^n = \om^{n,+} + \om^{n,-}$, where $\om^{n,+} = \om^{n,+}(t,x) = \om^{n,+}_0 ((\Phi^n)^{-1}(t,x))$ and 
$\om^{n,-} = \om^{n,-}(t,x) = \om^{n,-}_0 ((\Phi^n)^{-1}(t,x))$. 
Then, since $\om^{n,-}(t,\cdot) $ is a rearrangement of $\om^{n,-}_0(\cdot)$  and since, by the Dunford-Pettis theorem, $\om_0^{n,-}$ is uniformly integrable, so is $\om^{n,-}$, uniformly in time. Furthermore, $\om^{n,+} \geq 0$, while $\om^{n,-}\leq 0$. It follows that, passing to subsequences as needed, $\om^{n_p,+}$ converges weak-$\ast$ in $L^\infty\loc(\R_+;\bmp)$, to some $\om^+ \in L^\infty\loc(\R_+;\bmp)$, while, by the Dunford-Pettis theorem, $\om^{n_p,-}$ converges weak-$\ast$ in $L^\infty\loc(\R_+;L^1(\Om))$ to some $\om^- \in L^\infty\loc(\R_+;L^1(\Om))$, $\om^- \leq 0$. With this notation we find $\om = \om^+ + \om^-$. We can further assume that $|\om^{n_p}|$ converges weak$\ast$ in $L^\infty(\R_+;\bmp)$ to some $\mu$. 
We prove now that $\mu$ is a continuous measure. 
Since $\om^{n_p}$ is bounded in $L^\infty(\R_+;H^{-1}_{loc}(\Om))$, we have that $\om\in L^\infty(\R_+;H^{-1}_{loc}(\Om))$. Therefore $\om$ is a continuous measure. But $\om^-\in L^\infty_{loc}(\R_+;L^1(\Om))$ so $\om^-$ is continuous as a measure. Since $\om=\om^++\om^-$ is a continuous measure, we infer that $\om^+$ is a continuous measure too. Finally, from the bound $|\om^{n_p}|\leq \om^{n,+}-\om^{n,-}$ we deduce that $|\mu|\leq \om^+-\om^-$ so $\mu$ must be continuous. 

We claim now that
\begin{equation}\label{delort}
u^{n_p}_1u^{n_p}_2\to u_1u_2 \quad\text{and}\quad (u^{n_p}_1)^2-(u^{n_p}_2)^2\to (u_1)^2-(u_2)^2  
\end{equation}
in the sense of distributions of $\R_+\times\Om$. 
To show this, let $\B$ be a ball whose closure is included in $\Om$ and let $f\in C^\infty_c (\B;\R_+)$. We denote by $\psi^n$ the unique function  defined on $\B$ such that $u^n=\nabla^\perp\psi^n$ and $\int_{\B}\psi_n=0$. Let us define
\begin{equation*}
\tu^n=\nabla^\perp(f\psi^n)
\end{equation*}
and extend it with zero values on $\R^2\setminus\B$. We define in the same manner $\psi$ and $\tu$ associated to $u$. Given that $u^n$ is bounded in $L^\infty\loc([0,\infty);L^2(\B))$, the Poincar\'e inequality implies that $\psi^n$ is bounded in $L^\infty\loc([0,\infty);H^1(\B))$. Given that $u^{n_p}-u$ converges to 0 weak$\ast$ in $L^\infty\loc([0,\infty);L^2(\Om))$, every weak limit in $L^\infty\loc([0,\infty);H^1(\B))$ of $\psi^{n_p}$ is a function of vanishing mean on $\B$ and whose $\nabla^\perp$ is $u$. So such a weak limit is necessarily $\psi$. Having established that   $\psi^{n_p}\to\psi$ weak$\ast$ in $L^\infty\loc([0,\infty);H^1(\B))$, it is trivial to see that $\tu^{n_p}\to\tu$ weak$\ast$ in $L^\infty\loc([0,\infty);L^2(\R^2))$.


Moreover, we have that
\begin{equation*}
\curl\tu^n=\curl (fu^n+\psi^n\nabla^\perp f)
=f\om^n+2u^n\cdot\nabla^\perp f+\psi^n\triangle f  .
\end{equation*}
Clearly $f\om^n$ is bounded in $L^\infty\loc\bigl([0,\infty);L^1(\R^2)\bigr)$. Moreover, given that $u^n$ and $\psi^n$ are bounded in $L^2(\B)$ we infer that
\begin{equation*}
2u^n\cdot\nabla^\perp f+\psi^n\triangle f  \quad\text{is bounded in}\quad   L^\infty\loc\bigl([0,\infty);L^1(\R^2)\cap L^2(\R^2)\bigr).
\end{equation*}
In particular, $\curl\tu^n$ is bounded in $L^\infty\loc\bigl([0,\infty);L^1(\R^2)\bigr)$. Given that the weak limit of $|\om^{n_p}|$ is a continuous measure, the same holds true for $|\curl\tu^{n_p}|$.

Given the informations we have on $\tu^{n_p}$ we can now apply \cite[Theorem 6.3.1]{Chemin}, to deduce that 
\begin{equation*}
\tu^{n_p}_1\tu^{n_p}_2\to \tu_1\tu_2 \quad\text{and}\quad (\tu^{n_p}_1)^2-(\tu^{n_p}_2)^2\to (\tu_1)^2-(\tu_2)^2  
\end{equation*}
in the sense of distributions of $\R_+\times\R^2$. 

Given that $f\in C^\infty_c (\B;\R_+)$ is arbitrary we infer that \eqref{delort} holds true in the sense of distributions of $\R_+\times\B$. Since $\B$ is an arbitrary ball relatively compact in $\Om$, we infer that \eqref{delort} holds true in the sense of distributions of $\R_+\times\Om$.

Given \eqref{delort}, it is trivial to pass to the limit in the equation of $u^{n_p}$. Indeed, let us consider $\phi\in C^\infty_c ([0,\infty)\times\Om)$ be a divergence free test vector field. We multiply the equation of $u^{n_p}$ by $\phi$ and integrate to obtain
\begin{equation*}
\int_\Om u^{n_p}_0 \cdot \phi(0,\cdot)+\int_0^\infty\int_\Om  u^{n_p}\cdot \partial_t\phi+\int_0^\infty\int_\Om  (u^{n_p}\otimes  u^{n_p})\cdot\nabla\phi=0.
\end{equation*}
When we send $n_p\to\infty$, the first two terms on the left-hand side converge trivially to the desired limit. The second term passes to the limit thanks to \eqref{delort}, since
\begin{equation*}
(u^{n_p}\otimes  u^{n_p})\cdot\nabla\phi=  u^{n_p}_1u^{n_p}_2(\partial_1\phi_2+\partial_2\phi_1)+[(u^{n_p}_1)^2-(u^{n_p}_2)^2]\partial_1\phi_1.
\end{equation*}

We obtain that 
\begin{equation*}
\int_\Om u_0 \cdot \phi(0,\cdot)+\int_0^\infty\int_\Om  u\cdot \partial_t\phi+\int_0^\infty\int_\Om  (u\otimes  u)\cdot\nabla\phi=0
\end{equation*}
which is the formulation in the sense of distributions of the incompressible Euler equations in velocity form.
\end{proof}

\section{Weak vorticity formulation}

The aim of this section is to give a vorticity formulation which is equivalent to the velocity formulation. This is an extension of a result due to S. Schochet, see \cite{Schochet95}. More precisely, we will prove the following theorem. 

\begin{theorem}\label{equivth}
Let $u_0\in \L $ be divergence free, tangent to the boundary and assume that $\om_0 = \mbox{ curl } u_0 \in\bm$.
Let $u\in L^\infty\loc([0,\infty);\L)$ be a divergence free vector field, tangent to the boundary,  and such that $\om=\curl u\in L^\infty\loc([0,\infty);\bm)$. Let $\gamma_j(t)$ be the circulation of $u(t)$ on $\Gamma_j$ (defined a.e. in $t$ thanks to Lemma \ref{defcirc}). Assume that $u$ is a weak solution  of the 2D incompressible Euler equations with initial data $u_0$.
Then $\gamma_j\in L^\infty_{\loc}([0,\infty))$ for all $j$. Moreover, $\om$ and $\gamma_j$, $j\in\{1,\dots,k\}$, verify the following identity:
\begin{multline}\label{weakvort}
\int_0^\infty\int_\Om\partial_t\varphi\d \om+ \sum_{j=1}^k\int_0^\infty \gamma_j(t) \partial_t\varphi(t,\cdot)\bigl|_{\Gamma_j}\d t+\sum_{j=1}^k\int_0^\infty\bigl(\gamma_j+\int_\Omega\w _j\d\omega\bigr)\int_\Omega  \bX_j \cdot\nabla\varphi \d\omega\\
+\int_0^\infty\iint_{\Om\times\Om}\frac{1}{2}
\Big( \nabla_x  \varphi (x) \cdot K (x,y) + \nabla_y\varphi(y)\cdot K(y,x)\Big)\d\omega(x)\d\omega(y)\\
+\int_\Om\varphi(0,\cdot)\d \om_0+\sum_{j=1}^k\gamma_{j}(0)\varphi(0,\cdot)\bigl|_{\Gamma_j}=0 
\end{multline}
for all test functions $\varphi\in C^\infty_c([0,\infty);\yinf)$.

Conversely, let $\om_0 \in \bm \cap \H$ and consider real numbers $\gamma_{j,0}$, $j=1,\ldots,k$. Suppose that $\om\in L^\infty\loc([0,\infty);\bm \cap \H)$ and $\gamma_j \in L^{\infty}_{\loc}([0,\infty))$ verify identity \eqref{weakvort} for all test functions $\varphi\in C^\infty_c([0,\infty);\yinf)$, with $\gamma_j(0)=\gamma_{j,0}$. Set 
$u=u(t)$ to be the vector field given by Proposition \ref{uomega} in terms of $\omega=\omega(t)$ and $\gamma_j=\gamma_j(t)$.
Then $u \in L^{\infty}_{\loc}([0,\infty));\L)$ and $u$ is a weak solution of the 2D incompressible Euler equations with initial data $u_0=u(0)$.

\end{theorem}

We observe from Proposition \ref{propestK} that the function $ \nabla_x  \varphi (x) \cdot K (x,y) + \nabla_y\varphi(y)\cdot K(y,x)$ is bounded and continuous except on the diagonal. On the other hand, since $\om=\curl u\in L^\infty\loc([0,\infty);\bmp\cap\H)$ the measure $\om\otimes\om$ attaches no mass to the diagonal. We infer that all the terms appearing in \eqref{weakvort} are well defined. 

We remark that \eqref{weakvort} can be written as an equation in $\mathscr{D}'(\R_+)$:
\begin{multline}\label{weakvortr}
\partial_t\int_\Om\theta\d\om+\sum_{j=1}^k\gamma'_j \theta\bigl|_{\Gamma_j}
=\sum_{j=1}^k\bigl(\gamma_j+\int_\Omega\w _j\d\omega\bigr)\int_\Omega  \bX_j \cdot\nabla\theta \d\omega\\
+\iint_{\Om\times\Om}\frac{1}{2}
\Big( \nabla_x  \theta (x) \cdot K (x,y) + \nabla_y\theta(y)\cdot K(y,x)\Big)\d\omega(x)\d\omega(y)
\end{multline}
for all test functions $\theta\in \yinf$. In fact, \eqref{weakvort} is equivalent to \eqref{weakvortr} plus the initial conditions $\om\bigl|_{t=0}=\om_0$ and $\gamma_j(0)$ given.

\begin{proof}
To show Theorem \ref{equivth}, we observe first that the map $\varphi\mapsto\nabla^\perp\varphi=\phi$ gives a one to one correspondence from $\yinf$  into $C^\infty_{c,\sigma} (\Om)$. Therefore, the velocity formulation in the sense of distributions (see relation \eqref{defveleq}) can be written as follows
\begin{equation}\label{weakvel}
\int_\Om u_0 \cdot \nabla^\perp\varphi(0,\cdot)+\int_0^\infty\int_\Om  u\cdot \partial_t\nabla^\perp\varphi+\int_0^\infty\int_\Om  (u\otimes  u)\cdot\nabla\nabla^\perp\varphi=0
\end{equation}
for all test functions $\varphi\in C^\infty_c ([0,\infty);\yinf)$. Using Lemma \ref{defcirc}, we have that
\begin{equation}\label{ident1}
 \int_\Om u_0 \cdot \nabla^\perp\varphi(0,\cdot)=
-\int_\Om \varphi(0,\cdot)\d\om_0-\sum_{j=1}^k\gamma_{j,0} \varphi(0,\cdot)\bigl|_{\Gamma_j}
\end{equation}
and
\begin{equation}\label{ident2}
\int_0^\infty\int_\Om  u\cdot \partial_t\nabla^\perp\varphi
= -\int_0^\infty\int_\Om \partial_t\varphi\d\om-\sum_{j=1}^k\int_0^\infty\gamma_j \partial_t \varphi(t,\cdot)\bigl|_{\Gamma_j}. 
\end{equation}

In order to show Theorem \ref{equivth} it suffices to prove the following proposition.
\begin{proposition}\label{propident}
Let $u\in \L$ be a divergence free vector field tangent to the boundary such that $\omega=\curl u\in\bm\cap\H$. Let  $\gamma_1,\dots,\gamma_k$ be the circulations of $u$ on each of the connected components of the boundary $\Gamma_1,\dots,\Gamma_k$. Let $\varphi\in \yinfb$. The following identity holds true:
\begin{equation}\label{identity}
\int_\Omega (u\otimes u):\nabla\nabla^\perp\varphi
=-\iint_{\Omega\times\Omega} H_\varphi(x,y) \d\omega(x)\d\omega(y)-\sum_{j=1}^k\al_j\int_\Omega  \bX_j \cdot\nabla\varphi \d\omega
\end{equation}
where
\begin{equation*}
 H_\varphi(x,y) = \frac{1}{2}
\Big( \nabla_x  \varphi (x) \cdot K (x,y) + \nabla_y\varphi(y)\cdot K(y,x)\Big)\quad\text{and}\quad 
\al_j=\gamma_j+\int_\Omega\w _j\d\omega.
\end{equation*}
\end{proposition}

Indeed, assume that this proposition is proved. Let us first suppose that $u_0\in \L $ is divergence free, tangent to the boundary and such that $\om_0 = \mbox{ curl } u_0 \in\bm$. Let $u\in L^\infty\loc([0,\infty);\L)$ be a divergence free vector field, tangent to the boundary,  and such that $\om=\curl u\in L^\infty\loc([0,\infty);\bm)$. Let $\gamma_j(t)$ be the circulation of $u(t)$ on $\Gamma_j$. From the definition of the circulation, see relation \eqref{defcirceq}, we immediately see that $\gamma_j\in L^\infty_{\loc}([0,\infty))$. Assume that $u$ is a weak solution of the 2D incompressible Euler equations with initial data $u_0$. Given that $u\in L^\infty\loc([0,\infty);\L)$ and $\om\in L^\infty\loc([0,\infty);\bm\cap\H)$ we have that $u(t)\in \L$ and $\om(t)=\curl u(t)\in \bm\cap\H$ for almost all times $t$. For those times $t$ we can apply Proposition \ref{propident} to obtain relation \eqref{identity} for $u(t)$ and $\om(t)$. Integrating this relation in time implies that 
\begin{multline}\label{ident3}
\int_0^\infty\int_\Om  (u\otimes  u)\cdot\nabla\nabla^\perp\varphi
= -\int_0^\infty\iint_{\Om\times\Om}\frac{1}{2}
\Big( \nabla_x  \varphi (x) \cdot K (x,y) + \nabla_y\varphi(y)\cdot K(y,x)\Big)\d\omega(x)\d\omega(y)\\
-\sum_{j=1}^k\int_0^\infty\bigl(\gamma_j+\int_\Omega\w _j\d\omega\bigr)\int_\Omega  \bX_j \cdot\nabla\varphi \d\omega\d t.
\end{multline}

Combining relations \eqref{ident1}, \eqref{ident2} and \eqref{ident3} shows that the left-hand side of \eqref{weakvort} is equal up to a sign to the left-hand side of \eqref{weakvel}. 

Conversely, let $\om_0 \in \bm \cap \H$ and consider  $\gamma_{j,0} \in \real$, $j=1,\ldots,k$. Suppose that $\om\in L^\infty\loc([0,\infty);\bm \cap \H)$ and $\gamma_j \in L^{\infty}_{\loc}([0,\infty))$ verify identity \eqref{weakvort} for all test functions $\varphi\in C^\infty_c([0,\infty);\yinf)$, with $\gamma_j(0)=\gamma_{j,0}$. Set $u=u(t)$ to be the vector field given by Proposition \ref{uomega} in terms of $\omega=\omega(t)$ and $\gamma_j=\gamma_j(t)$. We begin by noting that, by Proposition \ref{charH-1norm}, together with the hypothesis $\om\in L^\infty\loc([0,\infty); \H)$, it follows that $K[\om] \in  L^\infty\loc([0,\infty);L^2(\Om))$. Furthermore, since $\om\in L^\infty\loc([0,\infty);\bm  )$ and $\gamma_j \in L^{\infty}_{\loc}([0,\infty))$, we find that $u$ given by Proposition \ref{uomega}, 
\[u = K[\om] + \sum_{i=1}^k (\gamma_j + \int_{\Om} \mathbf{w}_j\,\mathrm{d}\om) \mathbf{X}_j,\] 
belongs to $ L^{\infty}_{\loc}([0,\infty));\L)$. Next, let $\phi \in C^{\infty}_c([0,\infty)\times\Om)$ be a divergence free test vector field. Clearly, there exists $\varphi \in C^{\infty}_c([0,\infty);\yinf)$ such that $\phi = \nabla^{\perp} \varphi$. With this notation the identities \eqref{ident1}, \eqref{ident2}
hold true, and Proposition \ref{propident} implies, as before, that \eqref{ident3} is valid as well. This trivially implies that $u$ is a 
weak solution of the 2D incompressible Euler equations with initial data $u_0=u(0)$.

This completes the proof of Theorem \ref{equivth}, once we establish Proposition \ref{propident}. 
\end{proof}

We prove now  Proposition \ref{propident}.
\begin{proof}[Proof of Proposition \ref{propident}]
We show first \eqref{identity} when $u$ and $\om$ are smooth. More precisely, we assume that $\om\in C^\infty_c (\Om)$. Then, by Proposition \ref{milton} we have that
\begin{equation*}
  u=K[\om]+\sum_{j=1}^k \al_j\X_j,\qquad \al_j=\gamma_j+\int_\Om\w_j\om.
\end{equation*}

We integrate by parts the left-hand side of \eqref{identity}:
\begin{equation*}
\int_\Om (u\otimes u):\nabla\nabla^\perp\varphi
=-  \int_\Om \dive(u\otimes u)\cdot\nabla^\perp\varphi
=-  \int_\Om u\cdot\nabla u\cdot\nabla^\perp\varphi
=\int_\Om \curl(u\cdot\nabla u)\varphi 
-\int_{\partial\Om} u\cdot\nabla u\cdot\hat n^\perp\varphi.
\end{equation*}
We claim that the boundary integral above vanishes. Indeed, let $C_j=\varphi\bigl|_{\Gamma_j}$ and write $u\bigl|_{\Gamma_j}=\beta\hat n^\perp$ with $\beta$ a scalar function defined on $\Gamma_j$. Then
\begin{equation*}
\int_{\Gamma_j} u\cdot\nabla u\cdot\hat n^\perp\varphi
= C_j \int_{\Gamma_j} \beta \hat n^\perp\cdot\nabla u\cdot\hat n^\perp
= C_j \int_{\Gamma_j} \hat n^\perp\cdot\nabla u\cdot u 
= \frac{C_j}{2} \int_{\Gamma_j} \hat n^\perp\cdot\nabla (|u|^2)
=0
\end{equation*}
since $\Gamma_j$ is a closed curve. Since $\curl(u\cdot\nabla u)=u\cdot\nabla\om$ we infer that
\begin{multline*}
 \int_\Om (u\otimes u):\nabla\nabla^\perp\varphi
=\int_\Om u\cdot\nabla\om\varphi 
=-\int_\Om u\cdot\nabla\varphi\om 
=-\int_\Om K[\om]\cdot\nabla\varphi\om 
-\sum_{j=1}^k\al_j\int_\Om \X_j\cdot\nabla\varphi\om \\
=-\iint_{\Om\times\Om}K(x,y)\cdot\nabla\varphi(x)\om(x)\om(y)\d x\d y 
-\sum_{j=1}^k\al_j\int_\Om \X_j\cdot\nabla\varphi\om. 
\end{multline*}
Symmetrizing the term involving the kernel of the Biot-Savart law $K(x,y)$ implies \eqref{identity}.

\bigskip

We prove now the general case. Let $\om^n$ be defined like in Proposition \ref{aproxom} and let us introduce
\begin{equation*}
u^n=  K[\omega^n]+\sum_{j=1}^k\bigl(\gamma_j+\int_\Omega\w _j\omega^n\bigr)  \bX_j.
\end{equation*}
From Proposition \ref{uomega}, we know that
\begin{equation*}
u=K[\omega]+\sum_{j=1}^k\bigl(\gamma_j+\int_\Omega\w _j\d\om \bigr)  \bX_j.
\end{equation*}

\medskip

Since $\omega^n$ and $u^n$ are smooth, we have that 
\begin{equation*}
\int_\Omega (u^n\otimes u^n):\nabla\nabla^\perp\varphi
=-\iint_{\Omega\times\Omega} H_\varphi(x,y) \omega^n(x)\omega^n(y)\d x\d y-\sum_{j=1}^k\al^n_j\int_\Omega  \bX_j \cdot\nabla\varphi \ \omega^n
\end{equation*}
where
\begin{equation*}
\al^n_j=\gamma_j+\int_\Omega\w _j\omega^n.
\end{equation*}
Given relations \eqref{conv2}, \eqref{conv1}, and \eqref{conv123}  it not difficult to pass to the limit in the above relation and obtain \eqref{identity}. First, from \eqref{conv1} we infer that $\al^n_j\to \al_j$ as $n\to\infty$ for all $j\in\{1,\dots,k\}$. Next, from \eqref{conv123} we infer that
\begin{equation*} 
u^n\to u \quad\text{strongly in }L^2\loc(\overline\Om).  
\end{equation*}
so that 
\begin{equation*}
 \int_\Omega (u^n\otimes u^n):\nabla\nabla^\perp\varphi  \stackrel{n\to\infty}\longrightarrow \int_\Omega (u\otimes u):\nabla\nabla^\perp\varphi.  
\end{equation*}

It remains to show that
\begin{equation}\label{convH}
 \iint_{\Omega\times\Omega} H_\varphi(x,y) \omega^n(x)\omega^n(y)\d x\d y
 \stackrel{n\to\infty}\longrightarrow 
\iint_{\Omega\times\Omega} H_\varphi(x,y) \d\omega(x)\d\omega(y).
\end{equation}

By Proposition \ref{propestK} we have that the function $H_\varphi(x,y)$ is bounded and continuous outside the diagonal. Since $\om^n$ is bounded in $L^1(\Om)$, there is a subsequence $\om^{n_p}$ such that $|\om^{n_p}|$ converges weakly in the sense of measures to some positive bounded measure $\mu$. From Proposition \ref{aproxom}, we know that the measure $\mu$ is continuous. Therefore, the mass that $\mu(x)\otimes\mu(y)$ attaches to the diagonal $D=\{(x,x);\ x\in\Om\}$ vanishes since
\begin{equation*}
\mu(x)\otimes\mu(y)(D)
=\iint_D1\d\mu(x)\d\mu(y)
=\int_\Om\mu(\{x\})\d\mu(x)
=0.  
\end{equation*}
From the definition of $\om^n$ it is easy to observe that 
$$
\int_{\Sigma_\ep}|\om^n|\leq |\omega|(\Sigma_{\ep+\frac1{2n}})
\quad\text{and}\quad 
\int_{B_R^c}|\om^n|\leq |\omega|(B_{R-1}).
$$
This implies that the sequence of measures $\om^n$ is tight. We conclude that the convergence stated in relation \eqref{convH} holds true for the subsequence $\om^{n_p}$ as a consequence of Lemma \ref{weakconvmeas}. This completes the proof of Proposition \ref{propident}.  
\end{proof}

\begin{definition} \label{weakvortformul}
Let $\om_0 \in \bm \cap \H$ and consider real numbers $\gamma_{j,0}$, $j=1,\ldots,k$. We say that the ($k+1$)-tuple $(\om, \gamma_1,\ldots,\gamma_k)$, is a solution of the weak vorticity formulation of the incompressible 2D Euler equations in $\Om$, with initial vorticity $\om_0$ and initial circulations $\gamma_{j,0}$, if  $\om\in L^\infty\loc([0,\infty);\bm \cap \H)$, if  $\gamma_j \in L^{\infty}_{\loc}([0,\infty))$, $j=1,\ldots,k$,  and if the  identity \eqref{weakvort}, with $\gamma_j(0)=\gamma_{j,0}$, holds true for every test function $\varphi\in C^\infty_c([0,\infty);\yinf)$.
\end{definition}

\begin{remark}
We have shown, in Theorem \ref{equivth}, that there is a one-to-one correspondence between weak solutions of the incompressible 2D Euler equations in $\Om$ and solutions of the weak vorticity formulation. Notice that the weak vorticity formulation allows for weak solutions for which the Kelvin circulation theorem is no longer valid along boundary components. 
\end{remark}

Next, our aim  is to collect as much information as possible on the circulations $\gamma_j$ of a weak solution. Suppose that $(\om, \gamma_1,\ldots,\gamma_k)$ is a solution of the weak vorticity formulation as in Definition \ref{weakvortformul}. Then the circulations $\gamma_j$, $j=1,\ldots,k$ can be expressed uniquely in terms of $\om$. Indeed, take in \eqref{weakvortr} successive test functions $\theta_{\ell} \in \yinf$, vanishing in a neighborhood of the boundary except in the neighborhood of one of the $\Gamma_{\ell}$, where it equals 1. We then obtain the following linear system of ODEs for the circulations: 

\begin{multline*} 
\partial_t\int_\Om\theta_{\ell}\d\om+ \gamma'_{\ell} 
=\sum_{j=1}^k\bigl(\gamma_j+\int_\Omega\w _j\d\omega\bigr)\int_\Omega  \bX_j \cdot\nabla\theta_{\ell} \d\omega\\
+\iint_{\Om\times\Om}\frac{1}{2}
\Big( \nabla_x  \theta_{\ell} (x) \cdot K (x,y) + \nabla_y\theta_{\ell}(y)\cdot K(y,x)\Big)\d\omega(x)\d\omega(y),
\end{multline*}

$\ell = 1,\ldots, k$. 

Clearly, this system of ODEs for the circulations has a unique solution. So we can express the circulations $\gamma_j$ in terms of $\omega$. Plugging the formula for the circulations in $\gamma_j$ in \eqref{weakvortr} and taking a test function $\theta\in C^\infty_c(\Omega)$, it is possible to obtain an equation in the sense of distributions for $\omega$ only. However, this equation is not very enlightening.

\begin{theorem} \label{en-e-thm}
Let  $(\om, \gamma_1,\ldots,\gamma_k)$ be a solution of the weak vorticity formulation as in Definition \ref{weakvortformul}. Then
the following conservation law holds true: 
\begin{equation} \label{en-e}
\int_\Om 1\d\om+\sum_{j=1}^k\gamma_j(t)=\int_\Om1\d\om_0+\sum_{j=1}^k\gamma_{j,0},
\end{equation}
for all $t\geq0$.
\end{theorem}

\begin{proof}
Let $\rho\in C^\infty_c (\R_+)$ and $\theta\in C^\infty_c (\R^2;[0,1]$ such that $\theta(x)=1$ if $|x|\leq 1$ and $\theta(x)=0$ if $|x|\geq 2$. Set $\theta_n(x)=\theta(x/n)$. We use \eqref{weakvort} with $\varphi=\rho(t)\theta_n(x)$ and $n$ sufficiently large to obtain
\begin{multline}\label{conservation}
\int_0^\infty\rho'\int_\Om\theta_n\d \om+ \sum_{j=1}^k\int_0^\infty \gamma_j \rho' +\sum_{j=1}^k\int_0^\infty\rho\bigl(\gamma_j+\int_\Omega\w _j\d\omega\bigr)\int_\Omega  \bX_j \cdot\nabla\theta_n\d\omega\\
+\int_0^\infty\rho\iint_{\Om\times\Om}\frac{1}{2}
\Big( \nabla_x  \theta_n (x) \cdot K (x,y) + \nabla_y\theta_n(y)\cdot K(y,x)\Big)\d\omega(x)\d\omega(y)\\
+\rho(0)\int_\Om\theta_n\d \om_0+\rho(0)\sum_{j=1}^k\gamma_{j,0}=0.
\end{multline}
 
We let now $n\to\infty$. By the dominated convergence theorem we have that
\begin{equation*}
 \int_0^\infty\rho'\int_\Om\theta_n\d \om\to  \int_0^\infty\rho'\int_\Om1\d \om
\quad\text{and}\quad 
\int_\Om\theta_n\d \om_0\to \int_\Om1\d \om_0.
\end{equation*}
Moreover
\begin{equation*}
 \Bigl|\int_\Omega  \bX_j \cdot\nabla\theta_n\d\omega\Bigr|
\leq \| \bX_j\|_{L^\infty}\| \nabla\theta_n\|_{L^\infty}\int_\Om1\d\om=\cO(1/n)
\end{equation*}
uniformly in time. Using relation \eqref{estK} we also have that
\begin{equation*}
| \nabla_x  \theta_n (x) \cdot K (x,y) + \nabla_y\theta_n(y)\cdot K(y,x)| \leq M_2\|\nabla\theta_n\|_{W^{1,\infty}(\Om)}=\cO(1/n)
\end{equation*}
uniformly in $t$ and $x$. We conclude that after passing to the limit $n\to\infty$ in \eqref{conservation} we get
\begin{equation*}
 \int_0^\infty\rho'(\int_\Om1\d \om+ \sum_{j=1}^k\gamma_j)
+\rho(0)(\int_\Om1\d \om_0+\sum_{j=1}^k\gamma_{j,0})=0. 
\end{equation*}
This completes the proof of \eqref{en-e}.

\end{proof}

\section{Weak solutions obtained by mollifying initial data}  

In this section we will specify those properties enjoyed by weak solutions which are weak limits of smooth solutions, obtained by smoothing out initial data. We begin by recalling a stronger notion of weak solution, introduced by two of the authors in \cite{LNX01} and \cite{LNX06}, called {\it boundary-coupled weak solution} for flows in a domain with a single, infinite, connected boundary. In that context a boundary-coupled weak solution is a solution satisfying the weak vorticity formulation \eqref{weakvort} for the wider class of admissible test functions  vanishing on the boundary, but not necessarily their derivatives (instead of vanishing in a  neighborhood of the boundary as assumed in relation \eqref{weakvort}).

\begin{definition} 
Let $\om_0 \in \bm \cap \H$ and consider real numbers $\gamma_{j,0}$, $j=1,\ldots,k$. We say that the ($k+1$)-tuple $(\om, \gamma_1,\ldots,\gamma_k)$, is a boundary-coupled weak solution of the incompressible 2D Euler equations in $\Om$, with initial vorticity $\om_0$ and initial circulations $\gamma_{j,0}$, if  $\om\in L^\infty\loc([0,\infty);\bm \cap \H)$, if  $\gamma_j \in L^{\infty}_{\loc}([0,\infty))$, $j=1,\ldots,k$,  and if the  identity \eqref{weakvort}, with $\gamma_j(0)=\gamma_{j,0}$, holds true for every test function $\varphi\in C^\infty_c([0,\infty);\yinfb)$. 
\end{definition}

We will see that a sufficient condition for a weak limit of smooth solutions to satisfy this stronger notion  is  that the circulations $\gamma_j$ be conserved. 

\begin{theorem} \label{thmsomething}
Let $\om_0 \in (\bmp + L^1(\Om))\cap \H$, and consider $\gamma_0^j \in \real$, $j=1,\ldots,k$. Let $(\om,\gamma_1,\ldots,\gamma_k)$ be a solution of the weak vorticity formulation as constructed  in Theorem \ref{theodelort}, \textit{i.e.} by smoothing out the initial vorticity and passing to the limit using the argument given by Delort. Then we have that 
\begin{enumerate}
\item \label{en-a}  $\gamma_j(t)\geq \gamma_{j,0}$ for almost every $t\geq 0$.
\item \label{en-c}  If $\gamma_j(t)\equiv\gamma_{j,0}$ for almost  all $t>0$ then the solution is a boundary-coupled weak solution. 
\end{enumerate}
\end{theorem}

\begin{proof}

To prove \eqref{en-a}, we go back to the notation and proof of Theorem \ref{theodelort}. We have that $\om^n$ is bounded in $L^\infty\loc(\R_+;\bmp + L^1(\Om))$, so there exists some $\omb\in L^\infty\loc(\R_+;\bmb)$ and a subsequence $\om^{n_p}$ such that $\om^{n_p}\rightharpoonup\omb$ in $L^\infty\loc(\R_+;\bmb)$ weak$\ast$. We also recall that $\om^n = \om^{n,+} + \om^{n,-}$, with $\om^{n,+} \geq 0$ and $\om^{n,-} \leq 0$, such that $\om^{n,+} $ is bounded in $L^\infty\loc(\real_+,\bmpb)$ and, passing to subsequences if needed, such that $\om^{n_p,-} \rightharpoonup \om^-$ weak-$\ast$ in $L^\infty\loc(\real_+,L^1(\Om))$. Hence, passing to further subsequences if necessary, there exists $\omb^+ \in L^\infty\loc(\R_+;\bmpb)$ such that $\om^{n_p,+}\rightharpoonup\omb^+$ weak$\ast$ in $L^\infty\loc(\R_+;\bmpb)$. Thus $\omb - \om^- = \omb^+ \in L^\infty\loc(\real_+,\bmpb)$.

We can assume without loss of generality that the subsequence $n_p$ is the same as the one considered in the proof of Theorem \ref{theodelort}. Since the solution $u^n$ is smooth, the Kelvin circulation theorem holds true so the circulation of $u^n$ on $\Gamma_j$ is constant in time, denoted by $\gamma_{j,0}^n$. We get from relations \eqref{circK}, \eqref{defuzero} and \eqref{defuzerobis} that
\begin{equation*}
\gamma_{j,0}^n=\gamma_{j,0}+\int_\Om\w_j\d\om_0-\int_\Om\w_j \om_0^n.  
\end{equation*}
We clearly have that
\begin{equation}\label{convcirc}
\gamma_{j,0}^n\to   \gamma_{j,0}\quad\text{as }n\to\infty.
\end{equation}

Using \eqref{defcirceq} we have that
\begin{equation} \label{needthisforfinal1}
\int_0^\infty\int_\Om \varphi\om^{n_p}+ \int_0^\infty\int_\Om u^{n_p}\cdot\nabla^\perp\varphi=-\sum_{j=1}^k \gamma_{j,0}^n \int_0^\infty\varphi\bigl|_{\Gamma_j} 
\end{equation}
for all $\varphi\in C^\infty_c ([0,\infty);\yinf)$. Passing to the limit $p\to\infty$ above yields 
\begin{equation} \label{needthisforfinal2}
\int_0^\infty\int_{\overline\Om} \varphi\d\omb+ \int_0^\infty\int_\Om u\cdot\nabla^\perp\varphi=-\sum_{j=1}^k \gamma_{j,0} \int_0^\infty\varphi\bigl|_{\Gamma_j}. 
\end{equation}
Using again Lemma \ref{defcirc} we have that
\begin{equation*}
\int_0^\infty\int_{\Om} \varphi\d\om+ \int_0^\infty\int_\Om u\cdot\nabla^\perp\varphi=-\sum_{j=1}^k \int_0^\infty\gamma_j\varphi\bigl|_{\Gamma_j}. 
\end{equation*}
Subtracting the two previous relations we get that
\begin{equation*}
\sum_{j=1}^k \int_0^\infty\bigl[\gamma_j- \gamma_{j,0}-\omb(\Gamma_j)\bigr]\varphi\bigl|_{\Gamma_j}=0.
\end{equation*}
We conclude that
\begin{equation} \label{grouch}
\forall j \quad   \gamma_j= \gamma_{j,0}+\omb(\Gamma_j)\quad\text{a.e. in }t.
\end{equation}
Now, $\omb(\Gamma_j) = (\omb - \om^-)(\Gamma_j)$, since $\om^-(\Gamma_j) = 0$, by virtue of $\om^- \in L^\infty\loc (\real_+;L^1(\Om))$. Since $\omb - \om^- \geq 0$, part \eqref{en-a} follows.

It remains to show part \eqref{en-c}. Let $\om^n=\curl u^n$ be the approximate solution constructed in the proof of Theorem \ref{theodelort}. We use the Kelvin circulation theorem together with Theorem \ref{equivth} applied to the smooth solutions $\om^n$ to obtain that  $\om^n$ verifies the formulation
\begin{multline}\label{weakvortomn}
\int_0^\infty\int_\Om\partial_t\varphi\om^n + \sum_{j=1}^k\int_0^\infty \gamma_{j,0}^n \partial_t\varphi(t,\cdot)\bigl|_{\Gamma_j}\d t+\sum_{j=1}^k\int_0^\infty\bigl(\gamma_{j,0}^n+\int_\Omega\w _j\om^n\bigr)\int_\Omega  \bX_j \cdot\nabla\varphi \om^n\\
+\int_0^\infty\iint_{\Om\times\Om}\frac{1}{2}
\Big( \nabla_x  \varphi (x) \cdot K (x,y) + \nabla_y\varphi(y)\cdot K(y,x)\Big)\om^n(x)\om^n(y)\\
+\int_\Om\varphi(0,\cdot)\om^n _0+\sum_{j=1}^k\gamma_{j,0}^n\varphi(0,\cdot)\bigl|_{\Gamma_j}=0
\end{multline}
for all test functions $\varphi\in C^\infty_c ([0,\infty);\yinf)$. Given that $\om^n$ is compactly supported in $\Om$ it is easy to see that \eqref{weakvortomn} holds true with $\varphi\in C^\infty_c ([0,\infty);\yinfb)$. Indeed, let $\varphi\in C^\infty_c ([0,\infty);\yinfb)$ and define $\varphi^k=\varphi \chi_k$ where $\chi_k$ was defined on page \pageref{chin}. Then it is trivial to see that after we put $\varphi^k$ as test function in \eqref{weakvortomn} and send $k\to\infty$ we obtain the desired relation. More precisely, if $k$ is sufficiently large then due to the compact support of $\om^n$ all terms in \eqref{weakvortomn} except for the second and last ones does not change if we replace $\varphi$ with $\varphi^k$. And the sum of the second term and the last term is vanishing.

We use next that $u$, respectively $u^n$, has circulation $\gamma_{j,0}$, respectively $\gamma_{j,0}^n$, around $\Gamma_j$ for all $j$. Recalling Lemma \ref{defcirc} we obtain that
\begin{equation}\label{eqcirc}
\int_0^\infty\int_\Om\varphi \d\om+\int_0^\infty\int_\Om u\cdot \nabla^\perp \varphi  
=\int_0^\infty\int_\Om\varphi \om^n+\int_0^\infty\int_\Om u^n\cdot \nabla^\perp \varphi +\sum_{j=1}^k(\gamma_{j,0}^n-\gamma_{j,0})\varphi\bigl|_{\Gamma_j} 
\end{equation}
for all $\varphi\in C^\infty_c ([0,\infty);\yinf)$. 

We show now that \eqref{eqcirc} holds true for all $\varphi\in C^\infty_c ([0,\infty);\yinfb)$. We can assume that $\varphi$ vanishes on the boundary of $\Om$. Indeed, if this is not the case then, with the notations introduced in the proof of Lemma \ref{defcirc}, $\varphi-\sum_{j=1}^kL_j\varphi_j$ vanishes on  the boundary of $\Om$ and  belongs to $\varphi\in C^\infty_c ([0,\infty);\yinfb)$. Moreover, $\varphi_j$ can be used as test function in \eqref{eqcirc}. So let $\varphi\in C^\infty_c ([0,\infty);\yinfb)$ vanishing on the boundary of $\Om$. We construct as above $\varphi^k=\varphi \chi_k\in C^\infty_c ([0,\infty)\times\Om)$ and use it as test function in \eqref{eqcirc}. We clearly have by the dominated convergence theorem that  
\begin{equation*}
 \int_0^\infty\int_\Om\varphi^k\d\om \to \int_0^\infty\int_\Om\varphi \d\om
\quad\text{and}\quad 
\int_0^\infty\int_\Om\varphi^k\om^n \to \int_0^\infty\int_\Om\varphi \om^n
\qquad \text{as } k\to\infty.
\end{equation*}
Moreover, given that $\varphi$ vanishes on $\partial\Om$ we have that $\nabla\varphi^k=\nabla \varphi \chi_k+\varphi \nabla\chi_k$ is uniformly bounded with respect to $k$ and has support included in the same bounded set. Therefore, $\nabla \varphi^k\rightharpoonup\nabla\varphi$ weakly in $L^2([0,\infty)\times\Om)$ with  support included in the same bounded set. Given that $u,u^n\in L^2\loc([0,\infty)\times\overline\Om)$ we infer that
\begin{equation*}
 \int_0^\infty\int_\Om u\cdot \nabla^\perp \varphi^k  \to  \int_0^\infty\int_\Om u\cdot \nabla^\perp \varphi  
\quad\text{and}\quad 
\int_0^\infty\int_\Om u^n\cdot \nabla^\perp \varphi^k \to \int_0^\infty\int_\Om u^n\cdot \nabla^\perp \varphi 
\qquad \text{as } k\to\infty.
\end{equation*}
We conclude that \eqref{eqcirc} holds true with $\varphi\in C^\infty_c ([0,\infty);\yinfb)$.

Given that $u^{n_p}-u\rightharpoonup 0$ weak$\ast$ in $L^\infty\loc([0,\infty);L^2(\Om))$ we obtain from \eqref{convcirc} and \eqref{eqcirc} that
\begin{equation}\label{convom}
 \int_0^\infty\int_\Om\varphi \om^{n_p} \to \int_0^\infty\int_\Om\varphi \d\om
\qquad \text{as } p\to\infty
\end{equation}
for all $\varphi\in C^\infty_c ([0,\infty);\yinfb)$. Let us fix some $\varphi\in C^\infty_c ([0,\infty);\yinfb)$. Recall that such a $\varphi$ can be used as test function in \eqref{weakvortomn}. Next, we let $n\to\infty$ in \eqref{weakvortomn}. Given the above relation and recalling that $\om^n\rightharpoonup\om$ weak$\ast$ in $L^\infty(\R_+;\bmp + L^1(\Om))$ we observe that all terms in \eqref{weakvortomn} except for the term on the middle line converge to the expected limit. To pass to the limit in that term we use the same argument as in the proof of Proposition \ref{propident} (more precisely the proof of relation \eqref{convH}) except that now we also have a harmless time integral. The main point in the argument is that, by Proposition \ref{propestK}, the function $H_\varphi$ is bounded and continuous outside the diagonal for any $\varphi\in C^\infty_c ([0,\infty);\yinfb)$. In order to be able to apply Lemma \ref{weakconvmeas} and deduce that we can pass to the limit in  the term on the middle line of relation \eqref{weakvortomn}, we need to show two things. One is that the weak limit of $|\om^{n_p}|$ is a continuous measure. The other one is that the sequence $\om^n$ is tight. 


We already know from the proof of Theorem \ref{theodelort} that the weak limit of $|\om^{n_p}|$ is a continuous measure. Let us show now that the sequence $\om^n$ is tight. Since $\gamma_j(t)=\gamma_{j,0}$, we deduce from relation \eqref{en-e} that $\om^{n_p}(\Om)=\om(\Om)=\om_0(\Om)$. Therefore relation \eqref{convom} holds true with $\varphi=\xi$ (it is an equality in this case). Since $\xi \chi_k\in C^\infty_0([0,\infty);\yinf)$, it can be used as test function in  \eqref{convom}. We finally infer that relation \eqref{convom} holds true with $\varphi=\xi (1-\chi_k)$. Now, since $\om^{n,-}$ converges weak-$\ast$ to $\om^-$ in $L^\infty_{loc}(\R_+;L^1(\Om))$ we have that
\begin{equation*}
 \int_0^\infty\int_\Om\xi (1-\chi_k) \om^{n_p,-} \to \int_0^\infty\int_\Om\xi (1-\chi_k)\d\om^-
\qquad \text{as } p\to\infty.
\end{equation*}
Subtracting this from relation \eqref{convom}  with $\varphi=\xi (1-\chi_k)$ implies that 
\begin{equation*}
 \int_0^\infty\int_\Om\xi (1-\chi_k) \om^{n_p,+} \to \int_0^\infty\int_\Om\xi (1-\chi_k)\d\om^+
\qquad \text{as } p\to\infty.
\end{equation*}
Since $\om^{n_p,+}$ and $\om^{n_p,-}$ are single-signed, the two above relations imply the tightness of $\om^{n_p,+}$ and $\om^{n_p,-}$ as measures. We conclude that the sequence $|\om^{n_p}|$ must also be tight because it is bounded by $\om^{n_p,+}-\om^{n_p,-}$ which is tight. This completes the proof.
\end{proof}

Our final result is that, if $\om_0 \in L^1(\Om) \cap \H$, then any solution of the weak vorticity formulation constructed as in Theorem \ref{theodelort} has no circulation defect. Thus, for weak solutions constructed as in Theorem \ref{theodelort}, the phenomena of circulation defect may happen only at the level of regularity of vortex sheets.

\begin{theorem} \label{lastone}
Let $\om_0 \in L^1(\Om) \cap \H$ and consider real numbers $\gamma_{j,0}$, $j=1,\ldots,k$. Then there exists a weak solution of the Euler equations with initial data $\om_0$, $\gamma_{j,0}$ for which the circulation of the velocity around each connected component of the boundary is conserved. Consequently,
there exists a boundary-coupled weak solution with initial data $\om_0$ and initial circulations $\gamma_{j,0}$.
\end{theorem} 

\begin{proof}
Let $\om$, $\gamma_j$ be constructed as in Theorem \ref{theodelort}. Recall the notation and proofs from Theorem \ref{theodelort}. We note that, since $\om_0 \in L^1(\Om)$, it follows that $\om_{0,n} \to \om_0$ weakly in $L^1(\Om)$, thus $\{\om_0^n\}$ is uniformly integrable. Now, $\om^n$ is a rearrangement of $\om_0^n$ and, therefore, it is also uniformly integrable, uniformly in time. We obtain that, passing to a subsequence $n_p$ as needed, $\om^{n_p} \to \om$ weak$\ast$ in $L^\infty\loc (\real_+;L^1(\Om))$. Therefore,
using the notation introduced in the proof of Theorem \ref{thmsomething}, we find that 
$\omb = \om$ in this case.

As in the proof of Theorem \ref{thmsomething}, we use the identities \eqref{needthisforfinal1} and \eqref{needthisforfinal2}, together with the fact that $\omb(\Gamma_j) = 0$, to conclude that
\begin{equation*}
\gamma_j (t) = \gamma_{j,0},
\end{equation*}   
for almost all $t>0$.

It follows from Theorem \ref{thmsomething}, part \eqref{en-c}, that $\omega$ is a boundary-coupled weak solution.

\end{proof}

\begin{remark}
It follows from the proof of Theorem \ref{lastone} that any weak solution constructed using the procedure of Theorem \ref{theodelort} (\textit{i.e.} any weak limit of the sequence of approximations) with initial vorticity
in $L^1(\Om)\cap\H$ satisfies the conclusions of Theorem \ref{lastone}. 
\end{remark}

\section{Net force on boundary components}

For smooth fluid flow, the net force exerted by the fluid on an immersed solid object with boundary $\Gamma$ is given by
\[\int_{\Gamma} p \hat{n} \,dS,\]
where $p$ is the scalar pressure and $\hat{n}$ is the unit exterior normal to the fluid at $\Gamma$. It is natural to ask whether it is
possible to make sense of this net force at the level of regularity of vortex sheets. This is the subject of the present section.

Clearly, vortex sheet flows are not regular enough to allow for traces of pressure at the boundary. To circumvent this, we will use the PDE itself. Let $u$ be a smooth solution of the incompressible 2D Euler equations in the exterior domain $\Omega$, whose boundary is $\partial\Omega = \cup_{i=1}^{k} \Gamma_i$. Let $\chi_i$ be a smooth cut-off function of a neighborhood of the $i$-th boundary component, $\Gamma_i$, so that $\chi_i \equiv 1$ in this neighborhood and $\chi_i$ vanishes in a neighborhood of all the remaining $\Gamma_j$, $j \neq i$. We introduce the following two smooth vector fields:
\begin{equation} \label{Phii1}
\Phi_i^1 = \Phi_i^1(x) = -\nabla^{\perp}(x_2\chi_i(x)),
\end{equation}
and 
\begin{equation} \label{Phii2}
\Phi_i^2 = \Phi_i^2(x) = \nabla^{\perp}(x_1\chi_i(x)).
\end{equation}
Next we note that, to compute the horizontal component of the net force exerted on the fluid by the $i$-th hole, 
\[\int_{\Gamma_i} p\hat{n}^1 \, dS, \]
we can take the inner product of the Euler equations with $\Phi_i^1$ and integrate on $\Omega$:
\[\int_{\Omega} [\partial_t u + (u\cdot \nabla )u] \cdot \Phi_i^1 \, dx = - \int_{\Omega} \nabla p \cdot \Phi_i^1 \, dx.\]
We then observe that, since $\Phi_i^1$ is divergence free, we have
\[- \int_{\Omega} \nabla p \cdot \Phi_i^1 \, dx = - \int_{\partial \Omega} p \Phi_i^1 \cdot \hat{n} \, dS = -\int_{\Gamma_i} p \hat{n}^1 \, dS.\]

Similarly, to compute the vertical component of the net force  exerted on the fluid by the $i$-th hole, 
\[\int_{\Gamma_i} p\hat{n}^2 \, dS, \]
it is enough to take the inner product of the Euler equations with $\Phi_i^2$ and integrate on $\Omega$. 

Note that, if we substitute $\Phi_i^j$ for $\Phi_i^j + \Phi$, where $\Phi$ is any smooth, divergence free vector field  tangent to $\partial \Omega$, then the calculations above do not change.

With this in mind, let $u \in L^{\infty}\loc([0,+\infty);L^2\loc(\overline{\Omega}))\cap C^0([0,\infty);\mathscr{D}'(\Om))$ and consider the functional
\[ \Phi \mapsto F_u(\Phi) = F_u(\Phi)(t) \equiv \partial_t \int_{\Omega} u \cdot \Phi \, dx - \int_{\Omega} [(u \cdot \nabla) \Phi] \cdot u \, dx,\]
where $\Phi \in \bcs$. Here $\bcs$ denotes the space of smooth divergence free vector fields, tangent to the boundary with compact support at infinity. We easily have that, for each test vector field $\Phi$, $F_u(\Phi) \in W^{-1,\infty}(0,+\infty)$. 
In light of our observations in the smooth setting, we wish to have $F_u(\Phi)$ depending on $\Phi$ only through its normal component at $\partial \Om$.
In other words, since $F_u(\Phi)$ is linear with respect to $\Phi$, we expect to have $F_u(\Phi)$ vanishing whenever $\Phi\in\bcs$. 
As we will see, this brings us naturally to the notion of boundary-coupled weak solution. We have defined boundary-coupled weak solution using the vorticity formulation of the Euler equations. However, in order to discuss the net force we are interested in, we require a definition of boundary-coupled weak solution arising from the velocity formulation. Such a definition was introduced by one of the authors in \cite{franckspaper}, in the setting of flow in the exterior of a single rigid body. We wish to extend this definition to flow in a domain exterior to $k$ holes, such as our $\Om$. 

One natural way to strengthen the notion of weak solution in Definition \ref{eulerweaksol} is to allow test vector fields $\phi$ who are merely tangent to $\partial\Om$, instead of compactly supported in $\Om$, i.e., substitute divergence free $\phi \in C^{\infty}_c([0,\infty) \times \Om)$ for $\phi \in C^{\infty}_c([0,\infty);\bcs)$. We will use such test vector fields to write an equivalent definition of boundary-coupled weak solution for the velocity formulation.

\begin{theorem} \label{equivdefbdrycoupled}

Let $\omega_0 \in \bm \cap \H$ and fix real numbers $\gamma_{j,0}$, $j=1,\ldots, k$. Let $\omega \in L^{\infty}\loc([0,\infty);\bm\cap\H)$, together with $\gamma_j \in L^\infty\loc([0,\infty))$, $j=1,\ldots,k$, be a boundary-coupled weak solution of the Euler equations with initial vorticity $\omega_0$ and initial circulations $\gamma_{j,0}$. Then $u=u(t)$ given by Proposition \ref{uomega} in terms of $\omega$ and $\gamma_j$ is a weak solution of the incompressible 2D Euler equations for which the integral identity in part (a) of Definition \ref{eulerweaksol} holds for divergence free test vector fields $\phi \in C^{\infty}_c([0,\infty);\bcs )$. 

Conversely, assume that $u \in L^{\infty}\loc([0,\infty);\L )\cap C^0([0,\infty);\mathscr{D}'(\Om))$ is a weak solution which satisfies the identity in Definition \ref{eulerweaksol}, item (a), for all divergence free $\phi \in C^{\infty}_c([0,\infty);\bcs )$. Suppose, additionally, that $\omega = \curl u \in L^{\infty}\loc([0,\infty);\bm)$. Let $\gamma_j$ denote the circulation of $u$ around $\Gamma_j$. Then $(\omega,\gamma_1,\ldots,\gamma_k)$ is a boundary-coupled weak solution of the 2D incompressible Euler equations. 

\end{theorem}

This is the analogue, for boundary-coupled weak solutions, of Theorem \ref{equivth}. The proof, which we omit, is essentially identical to that of Theorem \ref{equivth}. Indeed, the three main ingredients of that proof go through in this case. One main ingredient is that the map $\varphi\mapsto \nabla^\perp\varphi$ gives a one to one correspondence from $\yinfb$ into $\bcs$. The second ingredient is Lemma \ref{defcirc}; this goes through for test functions in $\yinfb$ as was observed in Remark \ref{defcircremark}. And the third main ingredient is  the identity proved in Proposition \ref{propident}, already stated for test functions in $\yinfb$.

We henceforth refer to $u$ as being a boundary-coupled weak solution (of the velocity formulation of the incompressible 
2D Euler equations) if the set of test vector fields allowed in item (a) of Definition \ref{eulerweaksol} includes all divergence free $\phi \in C^{\infty}_c([0,\infty);\bcs)$. 

Now, given a boundary-coupled weak solution $u$ it can be readily verified that the functional we defined above, $F_u(\Phi)$, vanishes when $\Phi \in \bcs$. Indeed, taking $\eta=\eta(t)\in C^{\infty}_c([0,\infty))$ and setting $\phi = \eta \Phi$, with $\Phi \in \bcs$, in the integral identity in Definition \ref{eulerweaksol}, (a), we get
\[\int\eta'(t)\int_{\Om}u\cdot\Phi \, dxdt + \int\eta(t) \int_{\Om} [(u \cdot \nabla )\Phi ]\cdot u \, dxdt + \int_{\Om} \eta(0)\Phi\cdot u_0\,dx = 0.\]
Now, 
\[ \int\eta'(t)\int_{\Om}u\cdot\Phi \, dxdt + \int_{\Om} \eta(0)\Phi\cdot u_0\,dx = - \int \eta\partial_t\int_{\Om} u \cdot \Phi \, dx dt,\]
where the right-hand-side derivative is interpreted in the sense of distributions. It follows that $F_u(\Phi) = 0 $ in $W^{-1,\infty}((0,\infty))$.

For each $i=1,\ldots,k$ set 
\[\mathbf{f_i} = (f_i^1,f_i^2) \equiv -(F_u(\Phi_i^1),F_u(\Phi_i^2)),\]
where $\Phi_i^1$ was defined in \eqref{Phii1} and $\Phi_i^2$ was given in \eqref{Phii2}.  
In view of our discussion it is reasonable to call $\mathbf{f_i}$ the {\it net force which the $i$-th hole exerts on the fluid}. We have established that, for smooth flow, this coincides with the classical calculation of this net force.

We say that the net forces $\mathbf{f_i}$ are well-defined if $F_u(\Phi)=0$ for all $\Phi \in \bcs$.

We have proved, with the discussion above together with Theorem \ref{equivdefbdrycoupled}, the following result.

\begin{corollary} \label{cornetforce}
Let $\omega_0 \in \bm \cap \H$ and fix real numbers $\gamma_{i,0}$, $i=1,\ldots,k$. Let $(\omega,\gamma_1,\ldots,\gamma_k)$ be a solution of the weak vorticity formulation with initial data $(\omega_0,\gamma_{1,0},\ldots,\gamma_{k,0})$. Let $u$ be given by Proposition \ref{uomega} in terms of $\omega$ and $\gamma_i$. Then the net forces $\{\mathbf{f_i}\}$, $i=1,\ldots,k$, are well-defined if and only if $(\omega,\gamma_1,\ldots,\gamma_k)$ is a boundary-coupled weak solution. 

If $(\omega,\gamma_1,\ldots,\gamma_k)$ is a solution constructed as in Theorem \ref{theodelort} for which $\gamma_i(t) = \gamma_{i,0}$, $i=1,\ldots,k$, $t>0$, then the net forces $\mathbf{f_i}$ are well-defined. In particular, when $\omega_0 \in L^1(\Omega)\cap \H$ there exists a weak solution with well-defined forces.
\end{corollary}

We conclude this section by observing that, as with the net forces on $\Gamma_i$, we can discuss the torque which the fluid exerts on each $\Gamma_i$. In the smooth setting this corresponds to 
\[\int_{\Gamma_i} p (x-\overline{x}_i)^{\perp}\cdot \hat{n} \, dS,\]
where 
\[\overline{x}_i = \frac{\int_{\Omega_i}   x\,dx}{\int_{\Omega_i}  dx}\]
is the centroid of the $i$-th hole $\Omega_i$.

For each $i=1,\ldots,k$ set
\[\Psi_i=\Psi_i(x)=\nabla^{\perp}\left(\frac{1}{2}|x-\overline{x}|^2\chi_i(x)\right).\]
At the level of regularity of vortex sheets the torque $\mathbf{\tau_i}$ across each boundary component $\Gamma_i$ corresponds to $-F_u(\Psi_i)$. The conclusions for the torques, analogous to those in Corollary \ref{cornetforce} for the net forces, hold true.

\section{Final remarks and conclusion}

The main point of the present article was to develop the vortex dynamics formulation of the incompressible 2D Euler equations in domains with boundary
with vortex sheet initial data. We chose to formulate our theory in exterior domains for fortuitous reasons -- this happened to be the problem we started
investigating, -- but most of what we have done can be easily extended for flows in the interior of a bounded region with smooth boundary. Indeed, let $\Omega$ now denote a connected, bounded domain in $\real^2$ with $\partial \Omega$ consisting of $k+1$ disjoint smooth curves, $\Gamma_0$ the outer boundary, and $\Gamma_j$, $j=1,\ldots,k$ the inner boundaries. Existence of a weak solution with vortex sheet initial data was already proved by Delort in \cite{Delort91}. There has been no weak vorticity formulation in this context. We denote by $\gamma_0$ the circulation around $\Gamma_0$, taken in the direction $\hat{n}_0^{\perp}$, where $\hat{n}_0$ is the exterior unit normal to $\Gamma_0$.

The appropriate adaptation of Definition \ref{weakvortformul} is:  

\begin{definition}
Let $\om_0 \in \bm \cap H^{-1}(\Omega)$ and consider real numbers $\gamma_{j,0}$, $j=0,\ldots,k$. We say that the ($k+2$)-tuple $(\om, \gamma_0,\gamma_1,\ldots,\gamma_k)$, is a solution of the weak vorticity formulation of the incompressible 2D Euler equations in $\Om$, with initial vorticity $\om_0$ and initial circulations $\gamma_{j,0}$, if  $\om\in L^\infty\loc([0,\infty);\bm \cap \H)$, if  $\gamma_j \in L^{\infty}_{\loc}([0,\infty))$, $j=0,\ldots,k$,  and if the  identity
\begin{multline*}
\int_0^\infty\int_\Om\partial_t\varphi\d \om+ \sum_{j=1}^k\int_0^\infty \gamma_j(t) \partial_t\varphi(t,\cdot)\bigl|_{\Gamma_j}\d t - 
\int_0^{\infty} \gamma_0 \partial_t\varphi \bigl|_{\Gamma_0} \d t + 
\sum_{j=1}^k\int_0^\infty\bigl(\gamma_j+\int_\Omega\w _j\d\omega\bigr)\int_\Omega  \bX_j \cdot\nabla\varphi \d\omega \\
+\int_0^\infty\iint_{\Om\times\Om}\frac{1}{2}
\Big( \nabla_x  \varphi (x) \cdot K (x,y) + \nabla_y\varphi(y)\cdot K(y,x)\Big)\d\omega(x)\d\omega(y)\\
+\int_\Om\varphi(0,\cdot)\d \om_0+\sum_{j=1}^k\gamma_j(0)\varphi(0,\cdot)\bigl|_{\Gamma_j} 
-\gamma_0(0)\varphi(0,\cdot)\bigl|_{\Gamma_0}  = 0
\end{multline*}
with $\gamma_j(0)=\gamma_{j,0}$, holds true for every test function $\varphi\in C^\infty_c([0,\infty);\yinf)$.
\end{definition}
 
The equivalence between this definition and the standard weak velocity formulation is a straightforward adaptation of Theorem \ref{equivth}; in fact, the proof is simpler because now the domain is bounded.
 Indeed, in a bounded domain, Propositions \ref{propK}, \ref{propestK}, \ref{propdefK} and \ref{milton} are valid as stated. As for Lemma \ref{defcirc}, the identity \eqref{defcirceq} must be substituted for
\begin{equation}\label{defcirceqbdd}
\int_\Omega\varphi\d \om + \int_\Omega v\cdot\nabla^\perp\varphi=-\sum_{j=1}^k\gamma_j\;\varphi\bigl|_{\Gamma _j} + \gamma_0\;\varphi\bigl|_{\Gamma _0}.
\end{equation}
For bounded domains we have $\H = H^{-1}(\Omega)$ and, also, Proposition \ref{unicdec} is trivial. The definition of $K[\omega]$ for a general $\omega \in H^{-1}(\Omega)$  
follows directly from elliptic regularity on a bounded domain. In addition, Proposition \ref{aproxom} remains valid, with the obvious adaptation in the definition of the 
cut-off function $\chi_n$, Propositions \ref{uomega} and \ref{charH-1norm} remain valid as stated, and Theorem \ref{theodelort} was already proved by Delort in \cite{Delort91}. 
The proof of the equivalence between the weak velocity and weak vorticity formulations in bounded domains, adapting Theorem \ref{equivth}, follows from what we have just discussed.  

Next, we note that the conservation law described in Theorem \ref{en-e-thm} for exterior domain flows has an analogous statement for bounded domain flows, namely, that 
\begin{equation} \label{en-e-bdd}
\int_\Om 1\d\om+\sum_{j=1}^k\gamma_j(t) -\gamma_0(t) =\int_\Om1\d\om_0+\sum_{j=1}^k\gamma_{j,0} - \gamma_{0,0}.
\end{equation}

However, in a bounded domain the function $\varphi = 1$ belongs to $\yinf$, so that taking such $\varphi$ in \eqref{defcirceqbdd} leads to each side of \eqref{en-e-bdd} vanishing, which trivializes this conservation law.

We conclude the discussion on flows in a bounded domain by considering flows obtained by smoothing out initial data. The first item in the statement of  Theorem \ref{thmsomething} must be substituted, in a bounded domain, for 
\begin{equation*} 
\gamma_j(t) \geq \gamma_{j,0}, \; j = 1,\ldots, k, \mbox{ and } \gamma_0(t) \leq \gamma_{0,0}
\mbox{ for almost every } t\geq 0.
\end{equation*}
As for item (b) of Theorem \ref{thmsomething}, it holds as stated, for bounded domain flows, as long as we include $j=0$ in the statement.
Of course, the statement of Theorem \ref{lastone} remains valid for bounded domain flows as well, and the analysis of the definition of the
net force on boundary components extends to this context without change.

All results in this article have been about the interaction of 2D vortex sheet flows with {\it compact} boundary components. It is natural to
investigate the relation between boundary circulation and vorticity in the presence of non-compact boundaries. The simplest such situation  
is flow in the half-plane. In \cite{LNX01}, the authors introduced a weak vorticity formulation for vortex sheet flows in the half-plane $\mathbb{H}$ and, also, the notion of boundary-coupled weak solution. In that same paper, existence of a boundary-coupled weak solution was established assuming the initial vorticity was a nonnegative measure in $H^{-1}(\mathbb{H})$. The corresponding existence result remains open for flows in domains with compact boundaries, basically due to not being able to exclude concentration of vorticity at each boundary component. In the present work, we proved that weak solutions obtained by mollifying initial data, for which the circulation along boundary components is conserved, are, in fact, boundary-coupled. For half-plane flows the circulation along the boundary is naturally defined as the integral of vorticity in the bulk of the fluid. Conservation of circulation in half-plane flows is, therefore, equivalent to conservation of mass of vorticity. We do not know whether mass of vorticity is conserved for vortex sheet flows in the half-plane, but we do not expect this to be the case. This raises the question of whether the converse to Theorem \ref{thmsomething} holds in general, i.e., if there exist boundary-coupled weak solutions which do not conserve circulation along boundary components.

The physically relevant solutions of the incompressible Euler equations are those obtained from the vanishing viscosity limit of solutions of the 
Navier-Stokes system. It is an important open problem whether solutions of the Navier-Stokes equations with a fixed initial condition, on a domain
with boundary, satisfying the no-slip boundary conditions, converge to a weak solution of the incompressible Euler equations in the vanishing viscosity    limit. See \cite{BT13} for the state of the art concerning this problem, and its connection with turbulence modelling. Weak solutions which are limits
of vanishing viscosity, if such solutions exist, are expected to exchange vorticity with the boundary, see the discussion in \cite{BT13}, and have vortex
sheet regularity, see \cite{LMNT08}. This means that the weak vorticity formulation in domains with boundary, as developed here, provides an appropriate 
context to seek vanishing viscosity limits. Weak solutions obtained by mollifying initial data and solutions obtained as vanishing viscosity limits ought to behave differently with respect to their interaction with boundaries. To illustrate this, we observe that, according to the proof of Theorem \ref{thmsomething},
weak solutions must satisfy identity \eqref{grouch}. In contrast, if we consider the approximate solution sequence $\{\omega^{\nu}\}$ given by (9.1)-(9.4)
in \cite{LMNT08}, the limit velocity is time-independent, given by $u_0$, but, by (9.49) and (9.54) in the same reference, $\omb(\{|x|=1\}) = \alpha(t-)$, which, if it does not vanish, implies that \eqref{grouch} is not satisfied in this limit. For initial vorticity in $L^1$, we have obtained weak solutions which conserve circulation around boundary components and for the half-plane and similar domains, existence of a weak solution in \cite{LNX01,LNX06} follows from the construction of an approximate solution sequence which does not concentrate vorticity at the boundary. In both cases, we obtain boundary-coupled weak solutions. Vanishing viscosity limits, on the other hand might not, maybe should not be boundary-coupled, which, in light of the discussion in Section 8, makes the discussion of solid-fluid interaction rather delicate.   

The system formed by the incompressible 2D Euler equations 
in the exterior of a compact rigid body together with the equations for the motion of the rigid body
under the fluid force was studied in \cite{GS12,franckspaper}. The existence of a weak solution for the coupled system was proved in two cases: initial vorticity in $L^p$, $p>1$ and symmetric body with symmetric vortex sheet data with a sign condition. These are two situations where boundary coupled weak solutions are known to exist, by the present work in the first case and by \cite{LNX06} in the second. 
Our work suggests two natural extensions of the results in \cite{GS12,franckspaper}- to the limit case $p=1$ and to motion with more bodies. 
The case of motion of a rigid body coupled with a general Delort solution is physically very interesting, but our analysis in Section 8 highlights the difficulty in defining the coupling, and makes this case a more challenging open problem.    

Another natural avenue for investigation related to the present work is to adapt those results, proved in this article for approximate solution sequences obtained by mollifying initial data, to other approximation schemes, such as numerical approximations, approximation by Euler-$\alpha$ solutions, and vanishing viscosity on Navier-Stokes solutions with Lions' free boundary or other, more general, Navier boundary conditions.
   
\section*{Appendix}

In this appendix we establish uniform estimates for the Green's function and for the Biot-Savart kernel of a general, smooth, bounded domain in the plane. Such estimates can be found in classical textbooks, see, for instance, Theorem 4.17 in \cite{Aubin82}, but the constants depend on the distance to the boundary. Here we show they are uniformly bounded in the whole domain.

\begin{proposition} \label{appGreenAndBS}
Let $U \subset \real^2$ be a smooth, bounded domain. Let $G_U=G_U(x,y)$ be the Green's function for the Dirichlet Laplacian in $U$ and set $K_U=K_U (x,y) = \nabla^{\perp}_x G_U(x,y)$. Then there exists $M = M(U)>0$ such that
\begin{equation} \label{appGreen}
|G_U(x,y)| \leq M(1+|\log|x-y||),
\end{equation} 
and 
\begin{equation} \label{appBS}
|K_U(x,y)| \leq \frac{M}{|x-y|},
\end{equation} 
for all $(x,y) \in U \times U$, $x \neq y$.
\end{proposition}

\begin{remark}
Estimate \eqref{appBS} is due to L. Lichtenstein, see \cite{L1918}. For convenience sake, we include a proof.

\end{remark}

\begin{proof}
The proof proceeds in several steps. 

First we establish the result when the domain $U$ is the unit disk $D=B(0;1)$. We have, for the disk,
\[G_D(x,y) = \frac{1}{2\pi}\log\frac{|x-y|}{|x-y^{\ast}||y|},\]
where $y^{\ast}=\ds{\frac{y}{|y|^2}}$. 
Now, it is easy to verify that 
\[\lim_{y\to 0} |x-y^{\ast}||y| = 1,\]
and this limit is uniform with respect to $x \in D$. Hence, there exists $0<r_0\leq 1$ such that $1/2 < |x-y^{\ast}||y| < 2$ for all $y\in D$ such that $|y|<r_0$ and all $x\in D$. Therefore, if $(x,y) \in D\times D$, $|y|<r_0$, then \eqref{appGreen} follows immediately. If, on the other hand, $(x,y) \in D\times D$ and $|y|\geq r_0$ then $|x - y^{\ast}| \leq 1 + 1/r_0$, so that
\[\left|\log\frac{|x-y|}{|x-y^{\ast}||y|}\right| \leq \log\frac{1+1/r_0}{|x-y|} - \log r_0 \leq M(1+|\log|x-y||).\]
This establishes \eqref{appGreen} if $U$ is the unit disk. 

Next we analyze $K_D$: we have
\[K_D(x,y) = \frac{1}{2\pi}\frac{(x-y)^{\perp}}{|x-y|^2} - \frac{1}{2\pi}\frac{(x-y^{\ast})^{\perp}}{|x-y^{\ast}|^2}.\]
Together with the fact that $|x-y|\leq |x-y^{\ast}|$, this trivially yields \eqref{appBS}.

The second step consists of extending these estimates to a general smooth, bounded, simply connected domain $D_0$. This can be done through a conformal map $T: D_0 \to D$. The existence of the biholomorphism $T$ follows from the Riemann mapping theorem; that it can be extended smoothly up to the boundary was established in \cite{BK1987}. Furthermore, since both $T$ and $T^{-1}$ are smooth diffeomorphisms and $\overline{D_0}$ is a compact set with smooth boundary, it follows that $T$ and $T^{-1}$ are globally Lipschitz and, therefore, there exist $C_1>0$ and $C_2>0$ such that
\[C_1|x-y|\leq|T(x) - T(y)| \leq C_2|x-y|,\]
for all $(x,y) \in D_0 \times D_0$. With this map in place we note that
\[G_{D_0}(x,y) = G_D(T(x),T(y)),\]
whence we obtain \eqref{appGreen}. Also,
\[K_{D_0}(x,y) = K_D(T(x),T(y))DT(x).\]
This yields, analogously, \eqref{appBS}, since $DT$ is uniformly bounded, which is a consequence of $T$ extending smoothly up to the boundary. 

The third step consists of examining the case of a(n unbounded) domain which is exterior to a single, connected, bounded, smooth domain. We begin with the domain exterior to the unit disk, $\Pi = \real^2 \setminus D$. In this case the Green's function is given by
\[G_{\Pi}(x,y)= G_D(x^\ast,y^\ast),\]
with $x^\ast=x/|x|^2$ and $y^\ast=y/|y|^2$, as before. A straightforward calculation yields
\[|x^\ast-y^\ast| = \frac{|x-y|}{|x||y|}\]
so that, since $|x|$, $|y| \geq 1$, it follows from \eqref{appGreen} for the unit disk that, for any $R>0$,
\[|G_{\Pi}(x,y)| \leq M(1+ |\log|x-y||),\]
for all $x$, $y \in \Pi$, $|x|,$ $|y| \leq R$, which is a local version of \eqref{appGreen}. Furthermore, the Biot-Savart kernel $K_{\Pi}$ can be computed from $G_{\Pi}$ and we find
\[|K_{\Pi}(x,y)| \leq \frac{C}{|x|^2}\frac{1}{|x^\ast-y^\ast|} = \frac{C|y|}{|x||x-y|}.\]
Hence we can also deduce a version of \eqref{appBS}, valid in bounded subsets of $\Pi$.

Let, now, $\Pi_0$ denote a general unbounded exterior domain, which is exterior to a single, connected, bounded, smooth obstacle. Using a biholomorphism $T:\Pi_0\to \real^2\setminus\overline{D}$, with smooth extension to the boundary, as was done in \cite{ILN03}, we obtain local versions of \eqref{appGreen} and \eqref{appBS} for this kind of domain from the estimates in the domain exterior to the unit disk.

The fourth step consists of analyzing the case of a smooth, bounded, connected domain $D_1$ in the plane with one obstacle, i.e., let $D_0$ be a bounded, simply connected smooth domain and let $\mathcal{V}_1 \subset D_0$ be another simply connected smooth domain. Set $D_1 = D_0 \setminus \mathcal{V}_1$ and denote $\Gamma_1 \equiv \partial\mathcal{V}_1$ and $\Gamma_0 \equiv \partial D_0$, so that $\partial D_1 = \Gamma_0 \cup \Gamma_1$, $\Gamma_0 \cap \Gamma_1 = \emptyset$. Let us begin by establishing \eqref{appGreen} in $D_1$. For each $y\in D_1$ fixed, consider the function $\psi = \psi(x)=G_{D_1}(x,y)-G_{D_0}(x,y)$. Recall that, by the maximum principle, both $G_{D_1}$ and $G_{D_0}$ are nonpositive everywhere. 
We have:
\[
\left\{
\begin{array}{l}
\Delta_x \psi = 0, x \in D_1,\\
\psi = 0, \;\;\; x \in \Gamma_0,\\
\psi = -G_{D_0} \geq 0, \;\;\;x \in \Gamma_1.
\end{array}
\right.
\]
Therefore, by the maximum principle we find $0 \leq -G_{D_1}(x,y) \leq -G_{D_0}(x,y)$, for all $(x,y) \in D_1 \times D_1$. Together with the estimate for $G_{D_0}$, this yields the desired result for $G_{D_1}$. 

We cannot use the maximum principle in such a simple manner for the components $K^j_{D_1}$, $j=1,2$, because we have no boundary information for either of $K_{D_0}$ or $K_{D_1}$ on any of the two boundaries. So, to establish \eqref{appBS} for $K_{D_1}$ we split the problem in three cases: (i) $y$ in the interior of $D_1$, far from either boundary, (ii) $y$ near the outer boundary $\Gamma_0$ and (iii) $y$ near the inner boundary $\Gamma_1$. In the first case, (i), we obtain \eqref{appBS} from Theorem 4.17 of \cite{Aubin82}, where this estimate is deduced with a constant depending on the distance to the boundary. Next we assume that $y$ is close to the outer boundary $\Gamma_0$, as  in (ii). Let $r>0$ and suppose that $y \in B(z_0;r/2)\cap D_1$ for some $z_0 \in \Gamma_0$. Assume that $r$ is sufficiently small so that $\partial B(z_0;2r)\cap D_1$ is a connected set. Let $\phi=\phi(x)$ be a smooth cut-off function for $B(z_0;r)\cap D_1$, i.e., $\phi \geq 0$, $\phi \equiv 1$ inside $B(z_0;r)\cap D_1$, $\phi \equiv 0$ outside $B(z_0;2r)\cap D_1$, $\phi \in C^{\infty}(D_1)$. Set $\varphi = \phi(x)G_{D_1}(x,y)$ and extend $\varphi$ to $D_0$ by setting it to vanish inside the domain bounded by $\Gamma_1$. We then have:
\[\Delta_x \varphi = \phi(x)\Delta G_{D_1}(x,y) + 2 \nabla \phi(x)\nabla G_{D_1}(x,y) + \Delta\phi (x)G_{D_1}(x,y) \]
\[=\phi(y)\delta_y (x) + 2 \nabla \phi(x)\nabla G_{D_1}(x,y) + \Delta\phi (x)G_{D_1}(x,y) .\]
Observe that $\phi(y) = 1$ and notice that the last two terms above are supported in $B(z_0;2r)\setminus B(z_0;r)$ and, hence, are smooth functions. Let $\psi = \varphi(x) - G_{D_0}(x,y).$ We have:
\[
\left\{
\begin{array}{l}
\Delta_x \psi = 2 \nabla \phi(x)\nabla G_{D_1}(x,y) + \Delta\phi (x)G_{D_1}(x,y) \equiv f_1, \;\;\;x \in D_0,\\
\psi = 0, \;\;\; x \in \Gamma_0.
\end{array}
\right.
\]
We now use the representation formula for $\psi $ in terms of $f_1$ to show that $\psi$ and its derivatives are bounded, uniformly with respect to $x$. We have:
\[\psi(x)= \int_{D_0} f_1(z)G_{D_0}(x,z)\,dz,\]
so that
\[\nabla^{\perp}_x\psi(x) = \int_{D_0} f_1(z)K_{D_0}(x,z)\,dz.\]
Now, $f_1$ is a bounded function and $K_{D_0}$ satisfies \eqref{appBS}, so it is easy to show that $\nabla^{\perp}_x\psi$ is bounded in $D_0$. Since $\psi = \phi(x) G_{D_1}(x,y) - G_{D_0}(x,y)$ and $\phi \equiv 1$ in $B(z_0;r)\cap D_1$, we can compute $K_{D_1}$ in terms of $\nabla^{\perp}_x\psi$ for $x$ near the boundary $\Gamma_0$:
\[K_{D_1}(x,y) = \nabla^{\perp}_x \psi(x) + K_{D_0}(x,y),\;\;\;\mbox{ for } x \in B(z_0;r)\cap D_1.\]
If $y \in B(z_0;r/2)\cap D_1$ and $x \in D_1$, $x \notin B(z_0;r)$ then $K_{D_1}(x,y)$ is trivially bounded, since $G_{D_1}(x,y)$ is smooth away from the diagonal $\{x=y\}$. Hence, we conclude that $|K_{D_1}(x,y)|\leq C(1 + 1/|x-y|) \leq C/|x-y|$ for all $x \in D_1$ and $y \in B(z_0;r/2)\cap D_1$, $z_0 \in \Gamma_0$. The analysis of case (iii) proceeds in a similar fashion, except that we must use the exterior domain $\Pi_0 \equiv \real^2 \setminus \mathcal{V}_1$ in place of $D_0$ everywhere in the argument above. We note that, despite the fact that estimates \eqref{appGreen} and \eqref{appBS} were only shown to hold in bounded subsets of an exterior domain, this is enough to estimate $K_{D_1}$ for $x\in D_1$ and $y$ near the inner boundary $\Gamma_1$, as the new auxiliary function $f_1$ will be compactly supported and as $x \in D_1$, $y$ near $\Gamma_1$, remain bounded. To conclude the proof of \eqref{appBS} in $D_1$ we argue by compactness of the boundaries $\Gamma_0$ and $\Gamma_1$ to obtain a constant $M=M(D_1)>0$ which is uniform near the boundary, i.e., $y \in B(z_0;r/2)\cap D_1$ and $x \in B(z_0;r)\cap D_1$.  This, together with Theorem 4.17 in \cite{Aubin82}, yields the desired estimate in $D_1$. 

The fifth and final step in the proof is to analyze the case of a general bounded, smooth, domain with $N$ holes. The proof proceeds by induction with respect to the number of holes by repeating the argument presented in the fourth step, when adding one hole at a time. 

This concludes the proof.
\end{proof}

\textbf{Acknowledgments.}  D. Iftimie, M.C. Lopes Filho, H.J. Nussenzveig Lopes and F. Sueur thank the Franco-Brazilian Network in Mathematics (RFBM) for its financial support. M.C. Lopes Filho acknowledges the support of CNPq grant \# 303089/2010-5. H.J. Nussenzveig Lopes thanks the support of CNPq grant \# 306331/2010-1 and FAPERJ grant \# 103.197/2012. This work was partially supported by FAPESP grant \# 07/51490-7, by the CNPq-FAPERJ PRONEX in PDE, by the CNRS-FAPESP project \# 22076 and by the PICS \# 05925 of the CNRS. D. Iftimie and F. Sueur thank UNICAMP for its generous hospitality, while M.C. Lopes Filho and H. J. Nussenzveig Lopes thank the Universit\'e de Lyon, where part of this work was completed. H.J. Nussenzveig Lopes also thanks the Universit\'e de Paris VI for its kind hospitality. Finally, the authors wish to acknowledge helpful discussions with J.-M. Delort, P. G\'erard and J. Kelliher.

\def\cprime{$'$} \def\cprime{$'$} \def\cydot{{\l}eavevmode\raise.4ex\hbox{.}}
  \def\cprime{$'$} \def\cprime{$'$}
  \def\polhk\#1{\setbox0=\hbox{\#1}{{\o}oalign{\hidewidth
  {\l}ower1.5ex\hbox{`}\hidewidth\crcr\unhbox0}}}

\bigskip

\begin{description}
\item[Dragoş Iftimie] Université de Lyon, CNRS, Université Lyon 1, Institut Camille Jordan, 43 bd. du 11 novembre, Villeurbanne Cedex F-69622, France.\\
Email: \texttt{iftimie@math.univ-lyon1.fr}\\
Web page: \texttt{http://math.univ-lyon1.fr/\~{}iftimie}
\item[Milton C. Lopes Filho] Instituto de Matem\'atica,
Universidade Federal do Rio de Janeiro,
Cidade Universit\'aria -- Ilha do Fund\~ao, 
Caixa Postal 68530, 
21941-909 Rio de Janeiro, RJ -- Brasil. \\
Email: \texttt{mlopes@im.ufrj.br} \\
Web page: \texttt{http://www.im.ufrj.br/mlopes}
\item[Helena J. Nussenzveig Lopes] Instituto de Matem\'atica,
Universidade Federal do Rio de Janeiro,
Cidade Universit\'aria -- Ilha do Fund\~ao, 
Caixa Postal 68530, 
21941-909 Rio de Janeiro, RJ -- Brasil. \\
Email: \texttt{hlopes@im.ufrj.br}\\
Web page: \texttt{http://www.im.ufrj.br/hlopes}
\item[Franck Sueur] UPMC Univ Paris 06, CNRS, UMR 7598, Laboratoire Jacques-Louis Lions, F-75005, Paris, France\\
Email: \texttt{fsueur@ann.jussieu.fr}\\
Web page: \texttt{http://www.ann.jussieu.fr/\~{}fsueur/}
\end{description}


\end{document}